\documentclass{article}
\usepackage{enumerate}
\usepackage{graphicx} 
\usepackage{amsthm, amssymb, amsmath}
\usepackage{mathrsfs}
\usepackage{tikz}
\usepackage{hyperref}

\newtheorem{thm}{Theorem}[section]
\newtheorem{cor}[thm]{Corollary}
\newtheorem{prop}[thm]{Proposition} 
\newtheorem{lem}[thm]{Lemma}

\theoremstyle{definition}
\newtheorem{defn}[thm]{Definition}

\theoremstyle{remark}
\newtheorem{remark}[thm]{Remark}
\newtheorem{conv}[thm]{Convention}

\newcommand{\define}[1]{\emph{#1}}

\newcommand{\into}{\hookrightarrow}
\newcommand{\GL}[1]{\mathrm{GL}\left(#1\right)}
\newcommand{\verts}[1]{{\mathbf{Verts}\left(#1\right)}}
\newcommand{\edges}[1]{{\mathbf{Edges}\left(#1\right)}}
\newcommand{\bb}[1]{\mathbb{#1}}
\newcommand{\bk}[1]{\left\langle {#1} \right\rangle}
\newcommand{\gen}[1]{\mathscr{G}_{#1}}

\newcommand{\Bass}{\mathbf{Bass}}

\newcommand{\onto}{\twoheadrightarrow}

\newcommand{\calA}{\mathcal{A}}
\newcommand{\bbA}{\mathbb{A}}

\newcommand{\calD}{\mathcal{D}}
\newcommand{\calW}{\mathcal{W}}

\title{A fast algorithm for Stallings foldings over virtually free groups}
\author{
Sam Cookson\\
Nicholas Touikan}

\begin{document}

\maketitle
\begin{abstract}
    We give a simple algorithm to solve the subgroup membership problem for virtually free groups given as a graph of finite groups. For a fixed virtually free group with a fixed generating set $X$, the subgroup membership problem is uniformly solvable in time $O(N\log^*(N))$ where $N$ is the sum of the word lengths of the inputs with respect to $X$. For practical purposes, this can be considered to be linear time. The algorithm is simple enough to work out small examples by hand and concrete examples are given to show how it can be used for computations in $\mathrm{SL}(2,\mathbb Z)$ and $\GL{2,\mathbb Z}$. We also give fast algorithm to decide whether a finitely generated subgroup of a virtually free group is free.
\end{abstract}
\section{Introduction}

The concept of folding a graph was introduced by Stallings in \cite{stallings_topology_1983}, and is one of the most fundamental operations in group theory from a topological perspective.  In that paper, Stallings folding provides an algorithm that takes a collection of elements $\{f_1,\ldots,f_r\}$ of a free group $F$ and produces a \emph{folded graph} which turns out to contain the core of the covering space corresponding to the subgroup $H=\langle f_1,\ldots,f_r \rangle \leq F$.

In \cite{kapovich_stallings_2002}, Stallings folding is presented within an algorithmic framework. A notable result from that paper is a solution to the \emph{subgroup membership problem} in a free group $F(a_1,\ldots,a_r)$: given an element $f$, represented by a word, and a subgroup $H$ represented by a tuple of words corresponding to a generating set, decide if $f \in H$. The algorithm amounts to treating the graph produced by the Stallings folding algorithm as a deterministic finite automaton and then to decide whether the word representing $f$ can be read off along a loop in this folded graph.
 
In this paper, we present a Stallings folding algorithm that constructs a folded graph that is suitable to solve the subgroup membership problem in a virtually free group. The algorithm is identical to the algorithm for free groups, only that it is preceded by steps we call \emph{saturations}. This is not the first subgroup membership algorithm for virtually free groups, but the algorithm we present here is so simple that we can easily analyze its running time. As an application, using the fact that the group of $2\times 2$ invertible integer matrices $\GL{2,\mathbb Z}$ is virtually free we have the following result that is proved in Section \ref{sec:gl2z}.
\begin{thm}\label{thm:intro}
    There is an algorithm that takes as input a $2\times 2$ invertible integer matrix $A$ and a tuple $(G_1,\ldots,G_m)$ of invertible $2\times 2$ integer matrices and decides if $A$ can be expressed as some product \[
     A= G_{i_1}^{\epsilon_1}\cdots G_{i_r}^{\epsilon_r}
    \] where $G_{i_j} \in \{G_1,\ldots, G_m\}$ and $\epsilon_j \in \{-1,1\}$. This algorithm operates in time $O(N\cdot \log^*(N))$, where $N$ is the sum of the absolute values of all the matrix coefficients that appear in the input and the running time is a count of the number of pointer and integer arithmetic operations.
\end{thm}

Here, $\log^*(n) = \min\{m :   n\leq 2^*(m)\}$ where $2^*(n+1) = 2^{2^*(n)}$ and $2^*(0)=1$.
For practical purposes, we can assume the running time given above to be linear. This fast running time is due to the analysis of disjoint set data structures initiated by Tarjan in \cite{tarjan_efficiency_1975}.

\begin{remark}\label{rmk:logstar}
Throughout the text $\log^*$ can always be replaced by even more slowly growing functions, such as the inverse Ackermann function. We refer an interested reader to the discussion preceding \cite[Theorem 4.3]{silva_finite_2016} for further details.
\end{remark}

\begin{conv}\label{conv:running_time}
  In this paper, unless it is said otherwise, such as in Theorem \ref{thm:intro}, running times will correspond to a count of pointer operations as well as integer operations involving integers of uniformly bounded size.
\end{conv}

\subsection{Main results}

A group $G$ is virtually free if it contains a finite index subgroup $F$ that is a free group. By \cite{karrass_finite_1973}, the class of virtually free groups coincides with the class of fundamental groups of graphs of groups with finite vertex groups. A graph of groups is a construction that assembles a group out of ``smaller'' vertex groups. Every graph of groups has a \emph{fundamental group} which is well-defined up to isomorphism. Precise definitions and details are given in Section \ref{sec:graphs_of_groups}. 

If $X$ is a set of symbols, we will denote by $(X^{\pm 1})^*$, the set of all strings of symbols from the set $\{y | y=x \textrm{~ or ~} y=x^{-1}, x\in X\}$ and treat formal inverses $x^{-1}, x \in X$ as single symbols. Given a word $W \in (X^{\pm 1})^*$ we will denote by $|W|$ the length of $W$. If $G=\langle X|R\rangle$ is a group given by a presentation then we denote by $\epsilon:(X^{\pm 1})^* \to G$ the natural semigroup homomorphism sending words in $(X^{\pm 1})^*$ to products of generators and their inverses. We call this homomorphism the \emph{evaluation map.}  Throughout the text, we will be mindful of the distinction between a group element and a word that represents it.

A word $W$ is said to be \emph{freely reduced} if it has no subwords of the form $xx^{-1}$ or $x^{-1}x$. When a word represents an element of the fundamental group of a graph of groups we will say it is reduced if it satisfies the requirements of Definition \ref{defn:reduced}.

\begin{thm}\label{thm:main}
  Fix $\mathbb A$, a graph of groups with finite vertex groups over a directed graph $A$, fix disjoint tuples of symbols $\gen{v}$ associated to each $v \in \verts A$, and a fix presentation for $\Bass(\bbA)=\bk{X,R}$ that all satisfies the requirements of Section \ref{sec:precise}. Let our alphabet of symbols be \[X = \edges A \cup\left(\bigcup_{v \in \verts A}\gen v\right).\] The algorithm given in Section \ref{sec:algo} will take as input a tuple of words $(W_1,\ldots,W_n)$ in $(X^{\pm 1})^*$ 
  and output $(\calA_F,v_0)$, a directed graph with edges labelled in $X$, with the property that $g \in \pi_1(\bbA,v)$ lies in the subgroup $H=\bk{\epsilon(W_1),\ldots,\epsilon(W_n)}$ if and only if any reduced word $U \in (X^{\pm 1})^*$ (in the sense of Definition \ref{defn:reduced}) that evaluates to $g$ can be read as the label of a closed loop in $\calA_F$ starting and ending at $v_0$.

  $\bbA$, $(\calA_F,v_0)$ can be constructed in time $O(N\log^*(N))$ from the input $(W_1,\ldots,W_n)$, where $N$ is the sum of lengths of the words in the input.
\end{thm}

The reader may want to peek at Section \ref{sec:sl2z-eg} for a concrete example of how $\bbA$ should be represented and at Sections \ref{sec:algo} and \ref{sec:fold-example} to see how the algorithm works along with an example. Theorem \ref{thm:main} that is proved in Section \ref{sec:correct} is an immediate consequence of Proposition \ref{prop:algo-time} and \ref{prop:read-reduced}. The actual running time of the folding algorithm depends on the graph of groups $\bbA$. Details of this dependency are given in Proposition \ref{prop:algo-time}.

By Theorem \ref{thm:main}, deciding subgroup membership amounts to treating the based directed labelled graph $(\calA_F,v_0)$ as a deterministic finite automaton and checking if a word representing an element is accepted by this automaton. Once the automaton is constructed, verifying this takes time linear in the length of the word.

Unfortunately, if a word representing the element is not reduced in the sense of Definition \ref{defn:reduced} then it may fail to be accepted even though it represents an element of the subgroup $H$. In applications, there seems to be now way to avoid this technicality. We overcome this issue by showing that we can produce a reduced form of a word in almost linear time. The following summarizes Propositions \ref{prop:quick-reduce}, \ref{prop:detect-free}, and \ref{prop:detect-equal} proved in  Section \ref{sec:applications}.

\begin{thm}
  Let $G=\langle Y |S \rangle$ be a fixed presentation of a virtually free group. Then there are algorithms that run in time $O(N\log^*(N))$ that do the following:
  \begin{itemize}
  \item Given as input a word $U \in (Y^{\pm 1})^*$, output a reduced word (in the sense of Definition \ref{defn:reduced}) $\overline U \in (X^{\pm 1})$ where $X$ is the generating set of $\Bass(\bbA) = \langle X | R \rangle$, for some graph of groups $\bbA$ with $G \approx \pi_1(\bbA)$, and where $N=|U|$.
  \item Given as input a word $U \in (Y^{\pm 1})^*$ and a tuple of words $(W_1,\ldots,W_n)$ in $(Y^{\pm 1})^*$, decide if $\epsilon(U) \in \langle \epsilon(W_n),\ldots,\epsilon(W_n)\rangle$, where $$N = |U|+|W_1|+\cdots+|W_n|.$$
 \item Given a tuple of words $(W_1,\ldots,W_n)$ in $(Y^{\pm 1})^*$, decide if $\langle \epsilon(W_n),\ldots,\epsilon(W_n)\rangle$ is a free group , where $$N =|W_1|+\cdots+|W_n|.$$
 \item Given two tuples of words $(W_1,\ldots,W_n), (U_1,\ldots,U_m)$ in $(Y^{\pm 1})^*$, decide if $\langle \epsilon(W_1),\ldots,\epsilon(W_n)\rangle = \langle \epsilon(U_1),\ldots,\epsilon(U_m)\rangle$, where $$N =|W_1|+\cdots+|W_n|+|U_1|+\cdots+|U_m|.$$
  \end{itemize}
\end{thm}

\subsection{Relationship to other work}
In \cite{touikan_fast_2006} a fast folding algorithm was given for free groups and will be used in this paper. The subgroup membership problem for amalgams of finite groups (a special, but interesting class of virtually free groups) was first solved algorithmically by Markus-Epstein in \cite{markus-epstein_stallings_2007} and it is shown in that paper to operate in quadratic time.

In \cite{kapovich_foldings_2005}, Kapovich, Miasnikov, and Weidmann give a general abstract algorithm that constructs folded graphs for subgroups of graphs of groups. No running time analysis is given, and the underlying data structures are directed graphs equipped with an elaborate labelling system. The algorithm we present is closely related to the one in \cite{kapovich_foldings_2005}. What we call $\bbA$-graphs in this paper could actually be seen as ``blow-ups'' of the $\bbA$-graphs in that paper (in fact, in the virtually free case, both data structures can be seen to encode equivalent information). 

One result that is not shown in this paper is that the folded graphs we produce are canonical. The reason for this is that the most sensible approach to this problem involves Bass-Serre theory, and this is already done in \cite{kapovich_foldings_2005}.  It is only a matter of translating between the two approaches.

In \cite{kharlampovich_stallings_2017}, Kharlampovich, Miasnikov, and Weil, give a folding algorithm that solves the membership problem for (relatively) quasiconvex subgroups of (relatively) hyperbolic groups. This provides yet another  solution to the membership problem for virtually free groups. However, the generality of this method precludes straightforward running time analyses.

Using deeper techniques from formal language theory in the vein of the Muller-Schupp theorem \cite{muller_groups_1983},  Silva, Soler-Escrivà, and Ventura in \cite{silva_finite_2016} construct so-called Stallings sections, which in turn give rise to automata that solve the subgroup membership problem. This result is very close to our main result, Theorem \ref{thm:main}. The main difference is that we rely on an explicit decomposition of a virtually free group as the fundamental group of a graph of finite groups. For a fixed virtually free group it is shown in \cite{silva_finite_2016} that such an automaton can be constructed in in time $O(n^3 \ln^*(n))$ (see Remark \ref{rmk:logstar}.) They also show that the existence of Stallings sections characterizes virtually free groups. 

Still more recently, Lohrey in \cite{lohrey_subgroup_2021}, gives a polynomial-time algorithm for the membership problem in $\GL{2,\mathbb Z}$, where elements are represented by power words. This result is complementary to the result in this paper in the following way. In Section \ref{sec:gl2z}, where we prove Theorem \ref{thm:intro}, we consider the complexity of representing matrices as products of generators from a fixed generating set. For example, using the notation of Section \ref{sec:gl2z} the matrix\[
  M =
  \begin{pmatrix}
    1 & 1024\\0&1
  \end{pmatrix}
  \]
  is equal to $E^{-1024}$ which, in this paper, we would represent by fully writing out \[
    \underbrace{a^{-1}eb^{-1}e^{-1}\cdots a^{-1}eb^{-1}e^{-1}}_{\textrm{1024 factors}},
  \] for a total of $4\times 1024=4096$ symbols. Writing numbers in binary (or decimal) provides exponential compression and the polynomial time algorithm in \cite{lohrey_subgroup_2021} can read the input $M$ as $(a^{-1}eb^{-1}e^{-1})^{\texttt{10000000000}}$ (the exponent is written in binary) which only requires 16 symbols. Thus, in spite of running in polynomial time, the algorithm in \cite{lohrey_subgroup_2021} will provide an exponential speedup to the membership problem for certain inputs and contexts. One example of this is when we fix an upper bound on the number of matrices that generate our subgroups.

Although proving this is beyond the scope of the paper, it should be apparent to an expert that the folding algorithm we present also amounts to constructing the 1-skeleton of a core of a covering space $\rho: Y \to X$ (as in \cite{stallings_topology_1983}) corresponding to our subgroup, where $X$ is the graph of spaces (see \cite{scott_topological_1979}) constructed from a graph of finite groups. 

\subsection{Structure of the paper}

This paper is written to be as elementary as possible. Although we are able to get pretty far without using the Bass-Serre theory of groups acting on trees, we still need the concept of a graph of groups. In Section \ref{sec:graphs_of_groups}, we provide definitions of graphs groups, fundamental groups of graphs of groups, and how precisely these need to be encoded. In Section \ref{sec:A-graphs}, we define so-called $\bbA$-graphs and set the notation and formalism needed for the rest of the paper. In Section \ref{sec:folding_algo}, we describe the folding algorithm and analyze its running time. The real work starts in Section \ref{sec:correct}, where we show that the result of our simple saturation-then-folding algorithm actually has the advertised properties. The arguments only rely on Van-Kampen diagrams over the presentation $\Bass(\bbA) = \bk{X|R}$. In Section \ref{sec:applications} we prove some extra results using our machinery including Theorem \ref{thm:intro}.

\subsection*{Acknowledgements}
Part of the work in this paper was done by Sam Cookson while supported by an NSERC USRA. Nicholas Touikan is supported by an NSERC Discovery grant. The authors also wish to thank the anonymous referee for their careful reading of an earlier draft and for proposing numerous corrections and improvements.

\section{Graphs of groups}\label{sec:graphs_of_groups}

In this section, we set our notation for graphs of groups. For further details, we refer the reader to \cite{bogopolski_introduction_2008} for a contemporary introduction to Bass-Serre theory.

A \define{graph} $X$ consists of a set of vertices $\verts X$, a set of oriented edges $\edges X$, two maps $i : \edges X  \rightarrow \verts X$ and $t : \edges X  \rightarrow \verts X$, and a fixed-point free involution $\overline{\phantom{e}}: \edges X \rightarrow  \edges X$ satisfying $i(\overline{e}) = t(e)$. An orientation of a graph $X$ is a choice of one edge from each pair $\{e, \overline{e}\}, e \in \edges X$.  A graph where each edge is oriented is a \define{directed graph.}

\begin{defn}[Graph of groups]\label{defn:gog}
  A graph of groups $\bb A$ with underlying directed graph $A$ is
  obtained with the following additional data:
\begin{itemize}
\item To each vertex $v \in \verts A$, we assign a vertex group $\bb A_v$.
\item To each edge $e \in \edges A$, we assign an edge group $\bb A_e$ and require $\bbA_{\bar e}=\bbA_e$.
\item For each $e \in \edges A$, we have a pair of monomorphisms:
  \begin{align*}
    i_e:\bbA_e &\into \bbA_{i(e)}\\
    t_e:\bbA_e &\into \bbA_{t(e)}
  \end{align*}
  as well as the equalities $i_{\bar e}=t_e$ and $t_{\bar e}=i_e$.
\end{itemize}
\end{defn}

We have two ways of obtaining a group from a graph of groups. First, we have the \define{Bass Group}, which is defined by the following relative presentation:
\[\mathbf{Bass}(\bb A) = \left\langle
\left.\begin{array}{c}
\bb A_v, v \in \verts A;\\
e \in \edges A  
\end{array}\right.
\middle | 
\left.\begin{array}{c}
e\bar e =1, e \in \edges A;\\
\bar e i_e(g) e = t_e(g), e \in \edges A, g \in \bb A_e
\end{array}\right.\right\rangle\]

In the Bass Group, each edge in $\edges A$ is an element whose inverse is $\bar e$. 
\begin{conv}
    For the rest of this paper, we will assume that all graphs are directed. In particular, we will assume $\edges X$ only contains one element in the pair $\{e,\bar e\}$. Given $e \in \edges X$, we will write $\overline e = e^{-1}.$
\end{conv}

A \define{path in $A$} is a sequence of edges, \begin{equation}\label{eqn:path} 
e_1^{\epsilon_1}\cdots e_p^{\epsilon_p}
\end{equation} where $e_i \in \edges A$ and $\epsilon_i \in \{\pm 1\}$ and \begin{equation}\label{it:conn}
\textrm{for all $2\leq j\leq p$,  $t(e_{j-1}^{\epsilon_{j-1}})=i(e_{j}^{\epsilon_j})$}.
\end{equation} We further say that the path (\ref{eqn:path}) has \define{length} $p$. If $i(e_1^{\epsilon_1})=a$ and $t(e_p^{\epsilon_p})=b$ then we say the path is \emph{from $a$ to $b$}. A path from $b$ to $b$ is called a \emph{loop at $b$}.
 
An \define{$\bb A$--loop based at $b \in \verts A$} is a sequence of alternating of elements in $\bb A_v, v \in \verts A$ and edges in $\edges A$
  \begin{equation}
    \label{eq:1}
    a_0e_1^{\epsilon_1}a_1 \cdots e_p^{\epsilon_p}a_p
  \end{equation}
  Where the sequence $e_1^{\epsilon_1}\cdots e_p^{\epsilon_p}$ is a
  closed loop based at $b$, and where
  \begin{enumerate}
    \item $a_0,a_p \in A_b$, and
    \item $a_i \in A_{\tau(e_i^{\epsilon_i})}$, for $1\leq i \leq p$.
\end{enumerate}
$\bb A$-loops naturally correspond to elements in $\mathbf{Bass}(\mathbb A)$, and we have the following definition.

\begin{defn}[Fundamental group of a graph of groups]\label{defn:pi_1_gog}
    Let $b \in \verts A$. Then the \emph{fundamental group of $\bb A$ based at $b$}, denoted  $\pi_1(\bb A,v)$, consists of the subgroup of $\mathbf{Bass}(\bb A)$ generated by $\bb A$-loops based at $b$.
\end{defn}

For readers not familiar with the definition above, we have the following relation to the more known construction involving a choice of spanning tree in the underlying graph and where the remaining edges give rise to \emph{stable letters}. The following is a classical fact from Bass-Serre Theory.
\begin{thm}[See {\cite[Chapter 2, Theorem 16.5]{bogopolski_introduction_2008}}]
    Let $\tau \subset \edges A$ be a spanning tree and let $b \in \verts A$ be arbitrary. Then\[
    \mathbf{Bass}(\bb A)/\bk{\bk{e \in \tau}} \approx \pi_1(\bb A,b).   
    \]
\end{thm}

\subsection{Algorithmic specifications of graphs of groups.}\label{sec:precise}
Let $\bb A$ be a graph of groups. We will now give a specification to encode $\bb A$ that will be suitable for use in a computer.

For each $\bb A_v, v \in \verts A$, let $\gen v = (a_1^v,\ldots,a_{n(v)}^v)$ be a tuple of generators for $\bb A_v$ and consider every $a_i^w; w \in \verts A, 1\leq i \leq n(w)$ to be a distinct symbol. Furthermore for each $\bbA_v$ fix a presentation $\bbA_v = \langle \gen v| R_v\rangle$, where $R_v$ is a finite set of relations.

For each $h \in \bb A_e, e \in \edges A$, we take a fixed word $W(t_e(h))$ in $\gen{t(e)}^{\pm 1}$ representing $t_e(h) \in \bb A_{t(e)}$, and we take a fixed word $V(i_e(h))$ in $\gen{i(e)}^{\pm 1}$ representing $i_e(h) \in \bb A_{i(e)}$. 

We now take the set of generating symbols (or alphabet) \begin{equation}\label{eqn:precise_gens}
X = \edges A \cup \left( \bigcup_{v \in \verts A} \gen v\right).    
\end{equation}
In particular, we take the set of directed edges $\edges A$ to be a subset of our alphabet. For relations, we take the set of words\begin{multline}\label{eqn:BS-eqns}
R = \{e W(t_e(h) e^{-1} V(i_e(h))^{-1}: e\in \edges A, h \in \bb A_e \} \\ 
\cup \left(\bigcup_{v \in \verts A} R_v\right).
\end{multline}
It is immediate that $\mathbf{Bass}(\bb A) = \bk{ X | R}$. If $\bb A$ has a finite underlying graph and the $\bb A_v$s are all finite, then this presentation will also be finite. We will call the symbols in $\gen v \subset X$ \emph{vertex groups symbols}, the symbols in $\edges A \subset X$ \emph{edge symbols}, and we will call the relations of the form $e W(t_e(h)) e^{-1} V(i_e(h))^{-1}\in R$ \emph{Bass-Serre relations}.

The reader may note that \eqref{eqn:BS-eqns} is not economical: we do not need to include a Bass-Serre relation for every $h \in \bbA_e$; we need only consider a generating set of $\bbA_e$. It turns out that for this application, adding all of these relations does not substantially impact the running time of the algorithm, and it also makes it easier to describe. The reader may also remark that we could take $\gen v$ to be $\bb A_v$, since this latter group is finite. While this is true, and doing so again doesn't substantially impact the running time, we'll see in the next section that it makes working out small examples inconvenient.

\subsection{An example: $\mathrm{SL}(2,\mathbb Z)$}\label{sec:sl2z-eg}
It is a classical fact that $$\mathrm{SL}(2,\mathbb Z) = \left\{
  \begin{pmatrix}
    a&b\\c&d
  \end{pmatrix}
  \in M_2(\mathbb Z)
  \middle |
  ad-cb=1
\right\}$$ can be written as an amalgamated free product $\mathrm{SL}(2,\mathbb Z) = C_4*_{C_2}C_6$, where $C_n$ denotes the cyclic group of order $n$, realized as
\begin{eqnarray*}
    C_4 &=& \left\{
    \pm \begin{pmatrix}
    1&0\\0&1
    \end{pmatrix},
    \pm \begin{pmatrix}
    0&-1\\1&0
    \end{pmatrix}
    \right\}\\
    C_6 &=& \left\{
    \pm \begin{pmatrix}
    1&0\\0&1
    \end{pmatrix},
    \pm \begin{pmatrix}
    1&-1\\1&0
    \end{pmatrix},
    \pm \begin{pmatrix}
    0&1\\-1&1
    \end{pmatrix}
    \right\}\\
    C_2 &=& \left\{
    \pm \begin{pmatrix}
    1&0\\0&1
    \end{pmatrix}
    \right\},\\
\end{eqnarray*}
 see \cite[\S5]{dicks_groups_1989}. To illustrate our notation, consider the graph $A$ with $\verts A=\{u,v\}$ and $\edges A=\{e\}$ with $i(e)=u$ and $t(e)=v$. We set $a=
 \begin{pmatrix}
   0&-1\\1&0
 \end{pmatrix}$, $b=
 \begin{pmatrix}
   0&1\\-1&1
 \end{pmatrix},$ and $t =
 \begin{pmatrix}
   -1 & 0 \\ 0 & -1
 \end{pmatrix}$.

 We take abstract presentations $\bbA_u = \langle a | a^4\rangle, \bbA_v = \langle b | b^6\rangle$, and $\bbA_e = \langle t | t^2\rangle$, and we set $i_e(t)=a^2$ and $t_e(t)=b^3$. This is enough information to encode the graph of groups $\bbA$ depicted in Figure \ref{fig:sl2z-gog}.
 \begin{figure}[htb]
   \centering
   
   \begin{tikzpicture}[node distance = 0.5cm]
     \coordinate (eww) at (0,0);
     \coordinate (vee) at (2,0);
     \draw[fill = black] (eww) circle (0.1);
     \draw[fill = black] (vee) circle (0.1);
     \node (u) [left of = eww] {$u$};
     \node (A_u) [above of = eww] {$\langle a | a^4\rangle$};
     
     \node (v) [right of = vee] {$v$};
     \node (A_v) [above of = vee] {$\langle b | b^6\rangle$};

     \draw (eww) --node{$\blacktriangleright$} node[above]{$\langle t|t^2\rangle$}node[below]{$e$} (vee);
     \node (t) at (1,-0.5) {$t$};
     \node (a2) at (0,-0.5) {$a^2$};
     \node (b3) at (2,-0.5) {$b^3$};
     \draw[|->] (t) --node[below]{$i_e$} (a2);
     \draw[|->] (t) --node[below]{$t_e$} (b3);
     
   \end{tikzpicture}
   \caption{A graph of groups $\bbA$ for $\mathrm{SL}(2,\mathbb Z)$.}
   \label{fig:sl2z-gog}
 \end{figure}
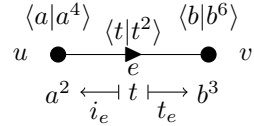

 For our algorithmic specification, we can take in addition to the graph $A$, $\gen u = \{a\}$ and $\gen v = \{b\}$ so that $X = \{a,b,e\}$ and $R = \{a^4,b^6,eb^3e^{-1}a^{-2}\}$ and $\Bass(\bbA) = \langle X | R \rangle$. We note that the symbol $t$ that we used to designate the generator of $\bbA_e$ doesn't actually occur in our alphabet $X$. Strings of symbols in $X^{\pm 1}$ represent elements $\Bass(\bbA)$, and those that are $\bbA$-loops based at $u$ are elements if $\pi_1(\bbA,u)$. For example $a^3eb^2e^{-1}a \in \pi_1(\bbA,u) \approx \mathrm{SL}(2,\mathbb Z)$,
 and $ab^2ebe \in \Bass(\bbA)\setminus \pi_1(\bbA,u)$.

\section{Words, evaluations, $\bb A$-graphs and the subgroups they define}\label{sec:A-graphs}

Let $X$ be a set of symbols. A \emph{directed $X$-labelled graph $\cal D$} is a directed graph equipped with a function $\mu:\edges{\cal D} \to X$, assigning to each edge a symbol in $X$.

Let $\rho = e_1^{\epsilon_1}\cdots e_p^{\epsilon_p}$ be path in $D$. Then we will abuse notation and denote the \define{label of $\rho$} as\[
\mu(\rho) = \mu(e_1)^{\epsilon_1}\cdots \mu(e_p)^{\epsilon_p} \in \left(X^{\pm 1}\right)^*.
\]

\begin{conv}\label{conv:gog}
    For the rest of the paper, we shall fix the (directed, unlabelled) graph $A$, the graph of groups $\bb A$, a generating set $X$, as given in (\ref{eqn:precise_gens}), and a set of relations $R$, as given in (\ref{eqn:BS-eqns}), all conforming to the specifications of Section \ref{sec:precise}. We assume $\Bass(\bbA)=\langle X | R \rangle$.
\end{conv}

Given some word $W=x_1^{\epsilon_1}\ldots x_l^{\epsilon_l} \in (X^{\pm 1})^*$, we define the  \define{evaluation of $\epsilon(W) \in \Bass(\bbA)$} to be the element\[
    \epsilon(W) = x_1^{\epsilon_1}\ldots x_l^{\epsilon_l} \in \Bass(\bbA)
\] where the right-hand side is seen as a long product of elements in $\Bass(\bbA)$. We introduce this concept to distinguish between different words, say the words $xx^{-1}x$ and $x$, even though they both evaluate to the same group element.

Given some path $\rho$ in $\cal D$, a directed $X$-labelled graph as above, abusing notation again, we define the \define{evaluation $\epsilon(\rho)$} to be the element\[
    \epsilon(\rho) = \epsilon(\mu(\rho)) \in \Bass(\bbA).
\]
 
A directed $X$-labelled graph $\cal A$ is called \define{an $\bb A$-graph} if there is a mapping 
\begin{align*}
    \verts{\cal A} &\to \verts A \\ v &\mapsto [v]
\end{align*}
such that the label of any loop $\ell$ in $\cal A$ based at $v \in \verts{\cal A}$ has an evaluation $\epsilon(\ell) \in \pi_1(\bb A,[v]).$ In particular, this implies that for any edge $f \in \edges{\cal A}$ with label $a^v_i \in \gen v$ with $i(f)=u,t(f)=w$, we have $[u]=[w]=v$. We also have that for any edge $f \in \edges{\cal A}$ with label $e \in \edges A$, if $i(f)=u,t(f)=w$, then $[u]=i(e),[w]=t(e)$ respectively.

Given an $\bb A$-graph $\cal A$ and a vertex $v \in \verts{\cal A}$, we wish to define two sets. First, the \define{language accepted by $(\mathcal A,v)$} is denoted by and defined as:
\[
L(\mathcal A,v) = \left\{ \mu(\ell) \in (X^{\pm 1})^* | \ell \textrm{ is a loop in $\cal A$ based at $v$}\right\}.
\] For $U,V \in (X^{\pm 1})^*$, we write $U\sim V$ if $U = _{F(X)} V$, i.e. if they define the same element of the free group $F(X)$ or, equivalently, if they can both be freely reduced to a common element. We use the following notation:
\[
\overline{L(\mathcal A,v)} = L(\mathcal A,v)/\sim
\]
Abusing notation, we also have the \define{subgroup defined by $(\mathcal A,v)$} that we denote as and define by:
\[
\pi_1(\mathcal A,v) = \left\{ \epsilon(\ell) \in \pi_1(\bb A,[v]) | \ell \textrm{ is a loop in $\cal A$ based at $v$}\right\}.
\]
We also note the following obvious equalities\[
\pi_1(\mathcal A,v) = \epsilon(L(\mathcal A,v)) = \epsilon\left(\overline{L(\mathcal A,v)}\right). 
\]

Finally, if $W \in (X^{\pm 1})^*$ evaluates to an element in $\pi_1(\bb A,v) \leq \Bass(\bb A)$, then we define its \emph{syllable decomposition} to be a factorization:\[
 W=W_0e_1^{\epsilon_1}\cdots e_n^{\epsilon_n} W_{n}
\] where $e_j \in \edges A$, $\epsilon_j \in \{-1,1\}$ and $W_j \in (\gen{u_j}^{\pm 1})^*$, where $u_j = i(e_{j+1})$ if $\epsilon_j=1$ and $u_j = t(e_{j+1})$ if $\epsilon_j = -1$, and where $W_n \in (\gen{u_0}^{\pm 1})^*$. We also must have that $e_1^{\epsilon_1}\cdots e_n^{\epsilon_n}$ is a loop starting and ending at the basepoint $v$.

It is clear from the specification of $X$ that this factorization of $W$ is well-defined. The following definition makes sense in the context of Convention \ref{conv:gog}.

\begin{defn}\label{defn:reduced}
    A word in $(X^{\pm 1})^*$ is said to be \emph{cancellable} if it is of the form
    \begin{equation}\label{eqn:reducible-subwords}
    e W  e^{-1}\textrm{ or } e^{-1} V e
    \end{equation} where $e \in \edges A$, $W \in (\gen{t(e)}^{\pm 1})^*$ and $\epsilon(W) \in t_e(\bb A_e) \leq \bb A_{t(e)}$ or $V \in (\gen{i(e)}^{\pm 1})^*$ and $\epsilon(V) \in i_e(\bb A_e) \leq \bb A_{i(e)}$. A word in $(X^{\pm 1})^*$ is said to be \emph{reduced} if it is freely reduced and contains no cancellable subwords.
\end{defn}

The following fact is immediate from the definitions.
\begin{lem}\label{lem:reduced-minmal}
    Every cancellable word evaluates to a word in $\bbA_v$ for some $v \in \verts A$. If $\epsilon(W) \in \pi_1(\bbA,v)$ but $W$ is not a reduced word, then there is a word $W'$ such that $\epsilon(W')=\epsilon(W)$ but with fewer symbols from $\edges A$. In particular every element of $\pi_1(\bbA,v)$ can be written as a reduced word.
\end{lem}

\subsection{Folding and morphisms}\label{sec:foldings}

Let $\calD,\calD'$ be two directed $X$-labelled graphs. A combinatorial function $f:\calD \to  \calD'$ that maps vertices to vertices and edges to edges such that $f(i(e))=i(f(e))$ and $f(t(e))=t(f(e))$ and $\mu(f(e))=\mu(e)$ for all $e \in \edges{\calD}$ is called a \emph{morphism}. We will also think of a morphism as being realized as a continuous map between the topological realizations of $\calD$ and $\calD'$. Obviously, a composition of morphisms is a morphism. If $\rho$ is a path in $\calD$ joining vertices $u,v$ then it has a well-defined image $f\circ\rho$, which is a path in $\calD'$ joining $f(u),f(v)$. By definition, we have $\mu(f\circ\rho) = \mu(\rho)$. When there is no risk of confusion, we will simply call $f\circ\rho$ the \emph{image of $\rho$ in $\calD'$}.

A \emph{folding at $v \in \verts \calD$} is a morphism $f:\calD \to \calD'$ that is as follows: at the vertex $v$, there is a pair of edges $e,e'$ such that $\mu(e)=\mu(e')$ and either $i(e)=i(e')=v$ or $t(e)=t(e')=v$. Then $\calD'$ is the quotient graph $\calD/\sim$ obtained by identifications $e\sim e'$, $i(e)\sim i(e')$, and $t(e)\sim t(e')$. The folding morphism is the quotient map $f:\calD \to \calD'$. We say that $\calD$ is \emph{completely folded} if there are no possible foldings, i.e. at every vertex, there is at most one adjacent edge with a given label and incidence. A \emph{folding process} is a sequence of folding morphisms that terminates in a completely folded graph.

If we name some vertex, say $v_0 \in \verts{\calD}$, then we will adopt the computer science convention of using the same notation to designate the image of that vertex throughout a folding process.  

\section{The folding algorithm}\label{sec:folding_algo}
    For this section, we fix a directed graph $A$ and a based  graph of groups $(\bb A,v)$. Additionally, we want:
    \begin{itemize}
        \item A set of symbols $X$ as given in \eqref{eqn:precise_gens} in Section \ref{sec:precise}, containing $\edges A$.
        \item The set of relations to be as given in \eqref{eqn:BS-eqns} in Section  \ref{sec:precise}.
        \item We will also assume that from the two items above, we will have constructed Cayley graphs $\Gamma_v = \mathrm{Cay}_{\gen v}(\bb A_v)$ for each finite vertex group with respect to our chosen generating set.
    \end{itemize}

    The purpose of this algorithm is to take a collection of reduced words defining elements that generate a subgroup $H$ of $\pi_1(\bb A,v)$ and output an $\bb A$-graph $(\mathcal A_F,v_F)$ with the property that if $W$ is any reduced word (in the sense of Definition \ref{defn:reduced}) in $(X^{\pm 1})^*$ representing an element $h \in H$ if and only if there is a unique loop $\ell_W$ in $\mathcal A_F$ based at $v_F$ with $\mu(\ell_W)=W.$ In other words, $W$, if it is reduced, can be read off directly in $\mathcal A_F$ starting and ending at $v_F$.

    In this section, we will give the algorithm and analyze its running time. The proof of correctness will be in the next section. We start by recalling some details about the classical folding algorithm.

\subsection{The fast folding algorithm in \cite{touikan_fast_2006}}\label{sec:orig-algo-description}

    The original Stalling folding algorithm implements a folding process on directed $X$-labelled graphs as described in Section \ref{sec:foldings}. The basic implementation of these graphs uses adjacency lists, i.e.  every vertex has a list of incident vertices and each edge has a pointer to an initial vertex, a terminal vertex as well as a label from $X$. The fast running time follows from two the use of two data structures:
    \begin{itemize}
    \item \textbf{Doubly linked lists.} Removal and insertion of a given element from a list as well as concatenation of lists uses a fixed number of pointer operations (irrespective of the size of the list.)
    \item \textbf{Disjoint sets.} Given a set $S$ we can think of equivalence classes or partitions of $S$. Consider a collection of equivalence classes where each class has a distinguished representative. \cite{tarjan_efficiency_1975} provides operations that given $x,y \in S$ will merge the equivalence classes to create the larger class $[x]\cup[y]$ as well as an operation to find the distinguished representative of the equivalence class containing an element $y$. It is shown that a sequence of $r$ of these \textit{merge} and \textit{find} operations has an amortized $O(r\log^*(r))$ running time.
    \end{itemize}

    In \cite{touikan_fast_2006} extra data and pointers are associated to the objects making up directed $X$-labelled graphs in order to take advantage of these data structures and associated operations.  It is shown that a folding operation, as described in Section \ref{sec:foldings}, can be realized using a fixed number of doubly linked lists and disjoint set operations. We now describe the folding process.

    Once an initial graph (a bouquet of generators) is constructed from the input, a (doubly linked) list called \texttt{UNFOLDED} is created and is arranged to contain all vertices that have two adjacent edges with the same label and incidence, i.e. vertices at which a folding can occur. This list may also contain extra vertices that do not have this property without impacting the correctness of the algorithm. The algorithm loops until \texttt{UNFOLDED} is empty. For each loop, a vertex $v$ from \texttt{UNFOLDED} is taken and $v$'s edge list is enumerated. Two things can happen:
    \begin{itemize}
    \item $v$ is found to have no adjacent edges with the same incidence and label and $v$ is removed from the list. This possibility involves $O(|X|)$ doubly linked list and disjoint set operations.
    \item Two edges $e,e'$ adjacent to $v$ are found to have the same label and incidence. If this occurs we stop the enumeration of the list, thus at most $O(|X|)$ doubly linked list and disjoint set operations will be used. The next step is to perform the folding operation (that uses a constant number of operations), a new vertex may be added to \texttt{UNFOLDED} and the total number of edges decreases.
    \end{itemize}

    Each time such a loop runs either a vertex gets removed from \texttt{UNFOLDED} or an edge gets deleted. The number of times a vertex can be removed  from \texttt{UNFOLDED} is no more than the original number of vertices in \texttt{UNFOLDED} plus the number of times some vertex gets added to the list (a vertex may be added and then removed multiple times). Since each time a vertex gets added to the list a folding must have occurred, decreasing the number of edges, we see that vertices are removed from \texttt{UNFOLDED} at most $V_0+E$ times, where $V_0+E$ is the sum of the number of vertices initially in the list \texttt{UNFOLDED} and the number of edges of the initial graph. It follows that the total running time is \begin{equation}\label{eqn:orig-bound}O\left(|X|(V_0+E)\log^*(|X|(V_0+E)))\right).
    \end{equation}

    \subsection{$\bb A$-graph folding algorithm}\label{sec:algo}

    The algorithm we present is the modification of the algorithm given in \cite{touikan_fast_2006} by adding the vertex and edge saturation steps.
    
    \begin{itemize}
        \item \textbf{Input:} A tuple $(W_1,\ldots,W_n)$ of words in $(X^{\pm 1})^*$ representing elements of $\pi_1(\bb A,v)$.
        \item \textbf{The algorithm:}
        \begin{enumerate}
            \item \textit{Form a bouquet of generators.}\label{it:form-bouquet} For each $W_i$, make a linear $X$-labelled directed graph along which we can read $W_i$. Identify all the endpoints of these linear graphs to create a graph that is a bouquet of circles $\mathcal A_0$ with a basepoint $v_0$
            \item \textit{Vertex saturation.}\label{it:vert-saturation} For each vertex $v \in \verts{\mathcal A_0}$, if $v$ is adjacent to an edge with a label in $\gen{u}$, we define $[v]=u \in \verts{A}$ and attach a copy of the Cayley graph $\Gamma_{[v]}$ to $v$. 
            If $v$ is adjacent to an edge $f$ such that $t(f)=v$ or $i(f)=v$ and $\mu(f)=e \in \edges A$, then attach a copy of the Cayley graph $\Gamma_{t(e)}$ or $\Gamma_{i(e)}$, respectively, to $v$. Call this new graph $\mathcal A_1$.
            \item \textit{Edge saturation.}\label{it:edge-saturation} For every edge $f \in \edges{\mathcal A_1}$ let $\mu(f) = e\in \edges A$. For each relation $r=eW(t_e(h))e^{-1}V(i_e(h))^{-1}$ involving $e$ occurring in $R$, attach a loop $\ell_r$ based at $i(e)$ whose label is $\mu(\ell_r)=r.$ Call the resulting graph $\mathcal A_2$.
            \item\textit{Stallings folding.}\label{it:fold} Apply the Stallings folding algorithm of \cite{touikan_fast_2006} to the $X$-labelled directed graph $\mathcal A_2$, using $X$ as the labelling alphabet and ensuring at initialization that all vertices in $\verts{\mathcal A_0} \subset \verts{\mathcal A_2}$ are in the list \texttt{UNFOLDED}. \textbf{Return} the folded graph $\mathcal A_F$.            
        \end{enumerate}
    \end{itemize}

    \subsection{Analysis of the running time}\label{sec:algo-time}
    We now prove the first half of the main result Theorem \ref{thm:main}.

    \begin{prop}\label{prop:algo-time}
      The algorithm given in Section \ref{sec:algo} runs in time \[O\left(|X|N(K+M+2))\log^*(|X|N(K+M+2))\right)\] where $N$ is the sum of the lengths of the input words, $M = \max_{v \in V(A)}\left(|\gen v|\cdot |\bbA_v|\right)$ and $K = \max_{r \in R}\left(|r|\right)$, where $X,R$ are the given generators and relations in the provided $\Bass(\bbA)) = \langle X | R \rangle.$
    \end{prop}
    \begin{proof}
      The graph $\calA_0$ has at most $N$ vertices and $N$ edges. Vertex saturation adds at most $N\cdot M$ edges and edge saturation adds at most $K\cdot N$ edges to the graph. It follows that $\calA_2$ has at most $N(K+M+1)$ edges. In $\calA_2$ the only vertices that may have multiple edges with the same label and incidence are the vertices in $\calA_0$. Thus by equation (\ref{eqn:orig-bound}), setting $V_0=N$ and $E=N(K+M+1)$, the algorithm will have a total running time of\[
        O\left(|X|N(K+M+2)\log^*(|X|N(K+M+2))\right).
        \]
Since the list \texttt{UNFOLDED} contained all possible vertices at which a folding could be performed in $\calA_2$, the algorithm will correctly terminate in a folded graph.
\end{proof}

    \subsection{An example}\label{sec:fold-example}
    We continue the example of $\mathrm{SL}(2,\mathbb Z)$ started in Section \ref{sec:sl2z-eg}, keeping the same notation. Consider $H = \langle a^3,aeb^2e^{-1}a^2\rangle \leq \pi_1(\bbA,u)$. The first step is to construct $\calA_0$ and then perform vertex and edge saturation. The resulting directed labelled graphs are depicted in figures \ref{fig:A_0+A_1} and \ref{fig:A_2+A_f}. In our example, since the edge groups have only one non-trivial element, edge saturations only involve attaching one ``rectangle'' to the initial vertex of each $\edges A$-labelled edge.
      \begin{figure}[htb]
        \centering
        \includegraphics[width=0.3\textwidth]{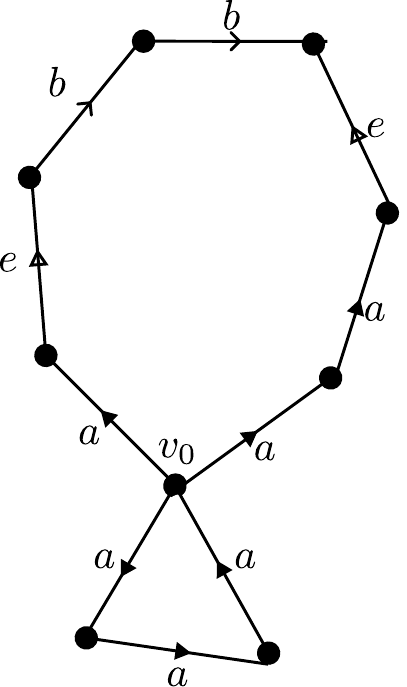}\hspace{2cm}
        \includegraphics[width=0.45\textwidth]{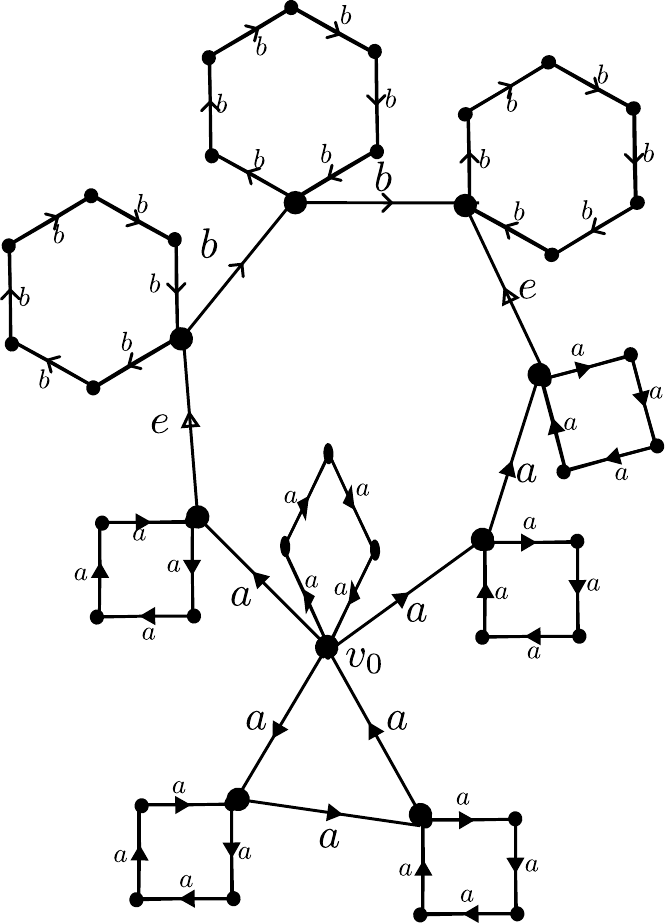}
        \caption{The $\calA_0$ from the input $(aaa,aebbe^{-1}a^{-1}a^{-1})$ and $\calA_1$, the result of attaching a copy of the appropriate Cayley graph at every vertex of $\calA_0$.}
        \label{fig:A_0+A_1}
      \end{figure}

      \begin{figure}[htb]
        \centering
        \includegraphics[width=0.5\textwidth]{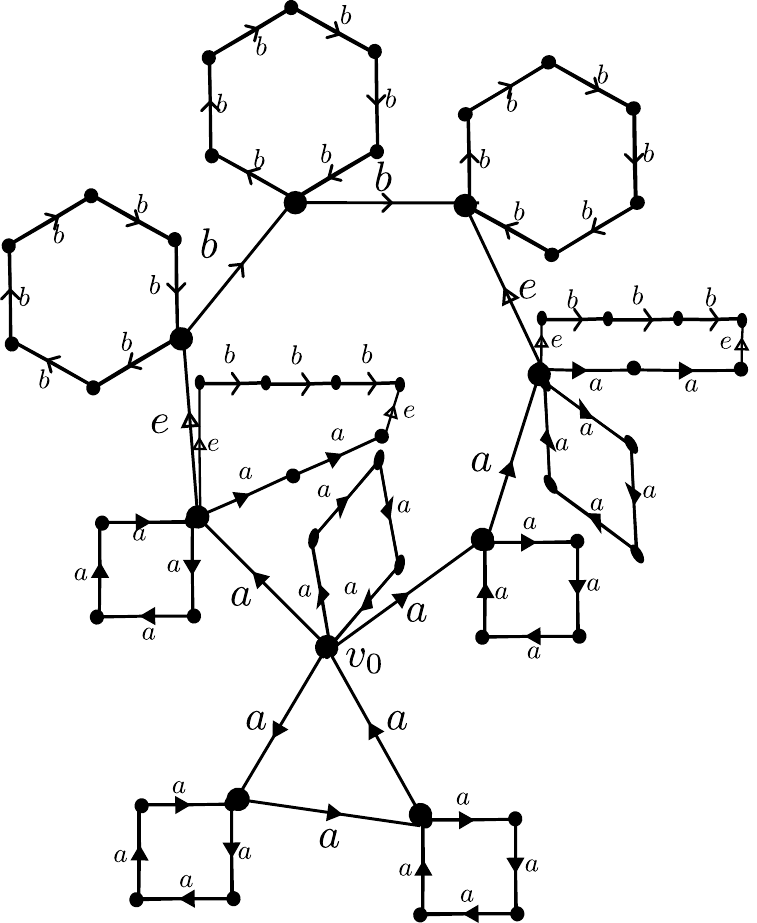}\hspace{2cm}
        \includegraphics[width=0.1\textwidth]{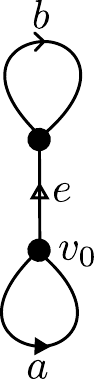}
        \caption{The graph $\calA_2$ and the final folded graph $\calA_F$, obtained by the standard folding algorithm.}
        \label{fig:A_2+A_f}
      \end{figure}

      In particular, from the graph $\calA_F$, we see that, for example, the element given by the word $W=a^2eb^3e^{-1}a^{-1}$, as we can read this word as a loop starting and ending at $v_0$. In fact, we see that $H = \pi_1(\bbA,u)$. In general, it may be that an unreduced word representing an element of the subgroup may not be readable in the final folded graph $\calA_F$.
\section{Proof of correctness}\label{sec:correct}
In this section, we show that the algorithm given in Section \ref{sec:algo} produces an output with the desired properties. The argument we give is purely combinatorial, but it is guided by the following idea: if the Cayley graphs we attach during vertex saturation were actually Cayley 2-complexes (i.e. we add 2-cells so that they are simply connected and admit a free action by the $\bbA_v$) and if the loops we add during edge saturation actually bounded 2-cells, then any path in $\calA_F$ could be homotoped to a path with reduced label.

\begin{conv}
    For this section, we will fix $\pi_1(\bbA,v)$, a tuple $(W_1,\ldots,W_n)$ of input words, as well as the subgroup $H \leq \pi_1(\bbA,v)$.
\end{conv}

\begin{prop}\label{prop:preservation}
    Let $\calA_F$ be the $\bbA$-graph produced by the algorithm in Section \ref{sec:algo}, and let $v_0$ be its basepoint. Let $(W_1,\ldots,W_n)$ be the tuple of input words. Then we have an equality of subgroups\[
    \pi_1(\calA_F,v_0) = \langle \epsilon(W_1),\ldots,\epsilon(W_n)\rangle =H \leq \pi_1(\bbA,v).
    \]
\end{prop}
\begin{proof}
    Let $\calA_0$ be the graph (the bouquet of generators) formed at step \ref{it:form-bouquet} of the algorithm. Then it is clear that $\pi_1(\calA_0,v_0)= \langle \epsilon(W_1),\ldots,\epsilon(W_n)\rangle$. Our goal is to show that this equality is preserved throughout the folding algorithm.

    \textit{Vertex saturation.} Step \ref{it:vert-saturation} involves attaching copies of Cayley graphs of vertex groups at various vertices of $\calA_0$. We will consider the effect of attaching a single one of these Cayley graphs. Let $\calA$ be an $\bbA$-graph, let $u \in \verts\calA$, let $[u]\in \verts A$ be its label, and consider the $\bbA$-graph $\calA'$ obtained as the quotient of the union\[
    \calA \cup (\Gamma_{[u]},u')/u\sim u',
    \] where $\Gamma_{[u]}$ is a copy of the given Cayley graph of $\bbA_{[u]}$ and $u'$ is some vertex of the Cayley graph. Let $\rho$ be a loop based at $v_0$ in $\calA'$. Then we can express $\rho$ as a concatenation\[
    \rho_1*\gamma_1*\cdots*\gamma_n*\rho_n,
    \] where the $\gamma_i$ are maximal subpaths that lie in $\Gamma_{[u]}.$ By maximality, the $\gamma_i$ must all be loops based at $u$, the vertex that is identified with the vertex $u'$. Now, the $\rho_i$ may not be loops, but their evaluations $\epsilon(\rho_i)$ are well-defined elements of $\Bass(\bbA)$ and the product $\epsilon(\rho_1)\epsilon(\gamma_1)\cdots\epsilon(\gamma_n)\epsilon(\rho_n) \in \pi_1(\bbA,v) \leq \Bass(\bbA)$. For any loop $\gamma_i$ in $\Gamma_{[u]}$, we must have that $\epsilon(\gamma_i) \in \bbA_{[u]}$ and that $\epsilon(\gamma_i)=1$. It follows that\[
        \epsilon(\rho_1*\gamma_1*\cdots*\gamma_n*\rho_n) = \epsilon(\rho_1)\cdots\epsilon(\rho_n),
    \] and since all the $\gamma_i$ are loops, the concatenation $\rho'=\rho_1*\cdots*\rho_n$ is well defined.

    Since there is an inclusion $\calA \subset \calA'$, we have $\pi_1(\calA,v_0) \leq \pi_1(\calA',v_0)$. We have just shown that for any $g \in \pi_1(\calA',v_0)$ represented by a loop $\rho'$ in $\calA'$ based at $v_0$, there is a loop $\rho$ in $\calA$ based at $v_0$ such that $\epsilon(\rho')=\epsilon(\rho)=g$. Thus $\pi_1(\calA',v_0) \leq \pi_1(\calA,v_0)$. Since vertex saturation involves performing this operation repeatedly, it follows that $\pi_1(\calA_1,v_0) = \pi_1(\calA_0,v_0)=H$.

    \textit{Edge saturation.} Step \ref{it:edge-saturation} involves attaching loops $\ell_r$, whose labels $\mu(r)$ are precisely relations in $R$, so that $\epsilon(\ell_r)=1$. An argument that is almost identical to the vertex saturation case establishes $\pi_1(\calA_2,v_0)=\pi_1(\calA_1,v_0)=H$.

    \textit{Folding.} Step \ref{it:fold} of the algorithm involves classical Stallings foldings. Our notation was chosen to match that of \cite{kapovich_stallings_2002}. By applying \cite[Lemma 3.4]{kapovich_stallings_2002} repeatedly we have $\overline{L(\calA_2,v_0)} = \overline{L(\calA_F,v_0)}$, thus\[
    \pi_1(\calA_F,v_0) = \epsilon(\overline{L(\calA_F,v_0)}) = \epsilon(\overline{L(\calA_2,v_0)})=\pi_1(\calA_2,v_0)=H,
    \] as required.
\end{proof}

The main difficulty when dealing with graphs of groups is that there are multiple equally valid ways to represent an element as a word. One way to overcome this is to introduce normal forms. In this paper, we chose instead to show that $L(\calA_F,v_0)$ contains all possible reduced forms of elements.

\begin{conv}
    When $\epsilon_i$s occur as exponents, they are either $1$ or $-1$.
\end{conv}

A first crucial observation is that in a fully folded directed $X$-labelled graph, a path $\rho$ has no backtracks if and only if its label $\mu(\rho)$ is a freely reduced word in $(X^{\pm 1})^*$. Consider the folded graph $\calA_F$, and for a given $v \in \verts A$, consider the subgraph consisting of all edges with labels in $\gen v$. We call the connected components of this subgraph \emph{the $v$-components of $\calA_F$}. It is clear that each vertex of $\calA_F$ lies in some $v$-component, and that deleting all edges with a label in $\edges A$ will disconnect $\calA_F$ into various $v$-components. 

It follows that given a path $\rho$ in $\calA_F$, we have an \emph{$\bbA$-decomposition}\[
    \rho = \alpha_1*e_1^{\epsilon_1}*\cdots e_n^{\epsilon_n}*\alpha_{n}
\] which is the unique decomposition with the $\alpha_i$ being maximal subpaths contained in $v$-components for an appropriate $v \in \verts A$, and where the $e_j$s have label $\mu(e_j) \in \edges A$. The paths $\alpha_i$s are called $\verts A$-components. We now give our two substitution lemmas.

\begin{lem}\label{lem:v-substition}
    Let $\rho = \alpha_1*e_1^{\epsilon_1}*\cdots*\alpha_i*\cdots*e_n^{\epsilon_n}\alpha_{n+1}$ be the $\bbA$-decomposition of a path in $\calA_F$. Let $\alpha_i$ lie in a $v$-component, and let $W_i' \in (\gen v^{\pm 1})^*$ be a word such that $\epsilon(\alpha_i) =_{\bbA_v} \epsilon(W_i')$. Then there is a path $\alpha_i'$ with the same endpoints as $\alpha_i$ such that $\mu(\alpha_i')=W_i'$ and \[
    \rho' = \alpha_1*e_1^{\epsilon_1}*\cdots*\alpha_i'*\cdots*e_n^{\epsilon_n}\alpha_{n}
    \] is also a path in $\calA_F$ with the same endpoints as $\rho$. Furthermore $\epsilon(\rho)=\epsilon(\rho')$.
\end{lem}
\begin{proof}
    Let $v_i \in \verts{\calA_F}$ be the initial vertex of $\alpha_i$, and set $v=[v_i]$. The preimages of $v_i$ in $\calA_2$ are either as vertices in the bouquet of generators in step \ref{it:form-bouquet}, vertices in one of the copies of Cayley graphs of $\bbA_v$ added in step \ref{it:vert-saturation}, or vertices in one of the loops added in step \ref{it:edge-saturation}. In all cases, when performing the folding process, any preimage of $v_i$ is eventually identified with the vertex of some Cayley graph added in step \ref{it:vert-saturation}. So we take a preimage $\hat v_i$ of $v_i$ that lies in some copy $\Gamma'_v$ of a Cayley graph of $\bbA_v$.

    By properties of the Cayley graph, since $\mu(\alpha_i)=_{\bbA_v} \mu(\alpha_i')$, there exist paths $\hat\alpha_i$ and $\hat\alpha_i'$  with labels $\mu(\alpha_i)$ and $\mu(\alpha_i')$  respectively both originating at $\hat v_i$ and terminating at the same vertex $\hat u_i$. Because $\calA_2 \to \calA_F$ is realized by a continuous map, $\hat\alpha_i$ is mapped to a path in $\calA_F$ starting at $v_i$, and because in a folded graph there is at most one path originating from a vertex with a given label, we have that $\hat\alpha_i$ is mapped to $\alpha_i$. It follows that $\alpha_i'$, the image of $\hat\alpha_i'$, has the same endpoints as $\alpha_i$, so $\rho'$ is also a path in $\calA_F$ with the same endpoints.

    Finally, the equality $\epsilon(\rho)=\epsilon(\rho')$ follows from the $\Bass(\bbA)$-equality\[
\epsilon(\rho')= \epsilon(\alpha_1)\epsilon(e_1)^{\epsilon_1}\cdots\epsilon(\alpha_i')\cdots\epsilon(e_n)^{\epsilon_n}\epsilon(\alpha_{n_1}),
    \] and the fact that $\epsilon(\alpha_i')=\epsilon(\alpha_i).$
\end{proof}
\begin{lem}\label{lem:e-substitution}
    Let $\rho = \alpha_1*{e_1}^{\epsilon_1}*\cdots\alpha_i*{e_i}^{\epsilon_i}*\alpha_{i+1}*\cdots*{e_n}^{\epsilon_n}\alpha_{n_1}$ be the $\bbA$-decomposition of a path in $\calA_F$, and let $\mu(e_i)W(t_{\mu(e_i)}(h))\mu(e_{i})^{-1}V(i_{\mu(e_i}(h))^{-1}$ be one of the relations in \eqref{eqn:BS-eqns}. Then the 1-edged path $e_i^{\epsilon_i}$ can be replaced by a path 
    \[
    \begin{cases}
        \nu*(e_i')^{\epsilon_i}*\omega & \textrm{ if $\epsilon_i=1$}\\ 
        \omega*(e_i')^{\epsilon_i}*\nu &\textrm{ if $\epsilon_i=-1$}\\ 
    \end{cases}
    \]
    where $\mu(e_i)=\mu(e_i')$ and $\mu(\omega) =W(t_{\mu(e_i)}(h))^{-\epsilon_i}$ and $\mu(\nu) = V(i_{\mu(e_i)}(h))^{\epsilon_i}$, giving a new path
    \begin{eqnarray*}
        \rho' &=& \alpha_1*{e_1}^{\epsilon_1}*\cdots*\alpha_i*\nu*e_i'*\omega*\alpha_{i+1}*\cdots*{e_n}^{\epsilon_n}\alpha_{n_1}\\
        &\textrm{or}&\\
        \rho' &=&\alpha_1*{e_1}^{\epsilon_1}*\cdots*\alpha_i*\omega*(e_i')^{-1}*\nu*\alpha_{i+1}*\cdots*{e_n}^{\epsilon_n}\alpha_{n_1},
    \end{eqnarray*}
    if $\epsilon_i =1$ or $\epsilon_1=-1$ respectively, that is also in $\calA_F$ with the same endpoints and $\epsilon(\rho)=\epsilon(\rho')$.
\end{lem}

\begin{proof}
  
  Take a preimage $\hat e_i$ of $e_i$ in $\calA_2$. $\hat e_i$ was either added in step \ref{it:form-bouquet} or in step \ref{it:edge-saturation}. We will suppose that $\hat e_i$ was added in step \ref{it:edge-saturation}, and we will assume $\epsilon_i=1$. The other cases follow from similar arguments.

  In this case, $\hat e_i$ is added as part of a loop $\ell_r$ with $$\mu(\ell_r)=\mu(e_i)W(t_{\mu(e_i)}(h_i))\mu(e_i)^{-1}V(i_{\mu(e_i)}(h_i))^{-1},$$ for some $h_i \in \bbA_{\mu(e_i)}$, as shown in Figure \ref{fig:e-sub}.

  \begin{figure}[h]
    \centering
    \includegraphics[width=\textwidth]{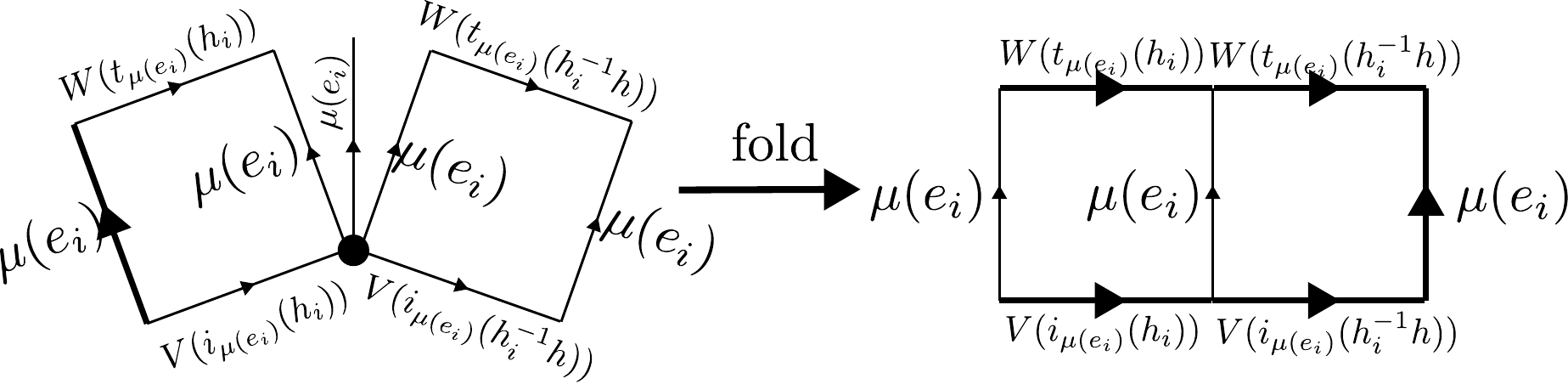}
    \caption{The two squares represent loops added in step \ref{it:edge-saturation}. The bold edge on the left is the preimage $\hat e_i$ of $e_i$ with label $\mu(e_i)$, and the bold path on the right is the new path $\hat\nu*\hat{e}_i'*\hat\omega$. The decorations are labels of edges or edge paths.}
    \label{fig:e-sub}
  \end{figure}

  An important feature of the folding algorithm is that the order in which we perform folds does not affect the final result. By step \ref{it:edge-saturation} of the algorithm, there will be another loop with label $\mu(e_i)W(t_{\mu(e_i)}(h_i^{-1}h))\mu(e_i)^{-1}V(i_{\mu(e_i)}(h_i^{-1}h))^{-1}$ attached to a corner of $\ell_r$ as shown in Figure \ref{fig:e-sub}, and we can immediately fold the three edges labelled $\mu(e_i)$ as shown in Figure \ref{fig:e-sub} in order to obtain an intermediate graph $\calA'$ such that there is a sequence of continuous maps $\calA_2 \onto \calA' \onto \calA_F$.

  In $\calA'$, looking again at Figure \ref{fig:e-sub}, we see that there is a path $\hat \nu *\hat e_i' * \hat \omega$ with the same endpoints as $\hat e_i$ where
\begin{eqnarray*}
  \mu(\hat\nu) &=& V(i_{\mu(e_i)}(h_i))V(i_{\mu(e_i)}(h_i^{-1}h))\\
  \mu(\hat e_i')&=&\mu(e_i)\\
  \mu(\hat\omega)&=&W(t_{\mu(e_i)}(h_i^{-1}h))^{-1}W(t_{\mu(e_i)}(h_i))^{-1}\\
\end{eqnarray*}

This path descends to a concatenation of path $\nu'*e_i'*\omega'$ with the same endpoints as $e_i$. Noting that in the respective vertex groups, we have  $$\mu(\nu')=V(i_{\mu(e_i)}(h_i))V(i_{\mu(e_i)}(h_i^{-1}h))\stackrel{\epsilon}{=}_{\bbA_{i(\mu(e_i))}}V(i_{\mu(e_i)}(h))$$ and $$\mu(\omega')=W(t_{\mu(e_i)}(h_i^{-1}h))^{-1}W(t_{\mu(e_i)}(h_i))^{-1}\stackrel{\epsilon}{=}_{\bbA_{t(\mu(e_i))}}W(t_{\mu(e_i)}(h))^{-1},$$ where here second equalities are  meant as ``have the same evaluation''. We can now apply Lemma \ref{lem:v-substition} to the paths $\nu'$ and $\omega'$ to obtain paths $\nu$ and $\omega$ with the same respective endpoints and desired respective labels.

\end{proof}

\begin{lem}\label{lem:reduction-sub}
  Let $\rho$ be a path in $\calA_F$ and suppose the label $\mu(\rho) \in (X^{\pm 1})^*$ is not freely reduced. Then there is another path $\rho'$ with the same endpoints as $\rho$ with $\epsilon(\rho') = \epsilon(\rho)$, but with $\mu(\rho')$ freely reduced.
\end{lem}
\begin{proof}
  If $\mu(\rho)$ is already freely reduced, there is nothing to show. Suppose $\mu(\rho)$ is not freely reduced. Then $\mu = Lx^{\pm 1}x^{\mp 1}R$ for words $L,R \in (X^{\pm 1})^*$ and some $x \in X$, and it follows that  $\rho = \lambda*e_i^{\epsilon_i}*e_{i+1}^{\epsilon_{i+1}}*\sigma$ where $\mu(e_i)=\mu(e_{i+1}) = x$.
  
  Because $\calA_F$ is folded, $e_i=e_{i+1}$ and $\epsilon_i=-\epsilon_{i+1}$, so $\rho' = \lambda*\sigma$ is a well-defined concatenation in $\calA_F$ with $\mu(\rho')$ freely equivalent to $\mu(\rho)$. Since $\rho'$ is shorter than $\rho$, we can repeat the process until we get the desired result.
\end{proof}

\begin{prop}\label{prop:read-reduced}
     $g \in H \leq \pi_1(\bbA,v)$ if and only if for any reduced word $W\in (X^{\pm 1})^*$ with $\epsilon(W)=g$, there is a loop $\ell_W$ based at $v_0$ in $\calA_F$ such that $\mu(\ell_w)=W.$
\end{prop}
\begin{proof}
  By Proposition \ref{prop:preservation} one implication follows immediately. It therefore remains to be shown that any reduced word representing some $g \in H$ is readable along a loop in $\calA_F$ based at $v_0$. We will argue using Van-Kampen diagrams.
  
  A Van Kampen diagram for a group presentation $\langle X | R \rangle$ is a simply connected CW 2-complex $D$ equipped with a fixed embedding in the plane. The 1-skeleton $D^{(1)}$ of $D$ has the structure of a directed $X$-labelled graph. We can read off words in $R^{\pm 1}$ along the boundaries of the 2-cells in $D$. By Van Kampen's Lemma, see \cite[\S V.1]{lyndon_combinatorial_2001}, for any word $W$ representing the identity in $\langle X | R\rangle$, there is a Van Kampen diagram $D$ with a vertex $d_0$ on its boundary such that the word $W$ can be read along the path that traces the boundary of $D$ when read in the clockwise direction.

  Let $g \in H$ be arbitrary, and let $W$ be an arbitrary reduced word, in the sense of Definition \ref{defn:reduced}, with $\epsilon(W)=g$. 
  By Proposition \ref{prop:preservation}, there exists a loop $\ell$ in $\calA_F$ based at $v_0$ such that $\epsilon(\ell) = g$. Let $\mu(\ell)=W'$. It follows that $W'W^{-1}$ represents the identity in $\langle X | R\rangle$ which is our presentation for $\Bass(\bbA)$. Consider the  Van Kampen diagram $D=D(W,W')$ for $W'W^{-1}$. It has two vertices $d_l,d_r$ (a ``leftmost'' and a ``rightmost'') such that for the path $\rho'$ that travels in $\partial D$ from $d_l$ to $d_r$  in the clockwise direction along the top of $D$, we have $\mu(\rho')=W'$ and for the path $\rho$ that travels in $\partial D$ from $d_l$ to $d_r$  in the counter-clockwise direction along the bottom of $D$ we have $\mu(\rho)=W$. This is depicted in Figure \ref{fig:van-kampen}.

    \begin{figure}[htb]
        \centering
        \includegraphics[width=\textwidth]{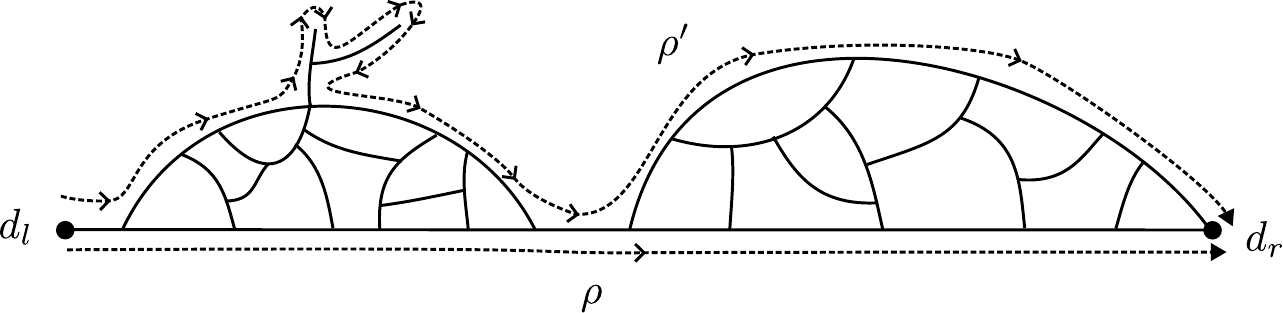}
        \caption{The Van Kampen diagram $D(W,W')$ for a pair of words representing the same group element. Note that this diagram has three degree 1 vertices, one of which is $d_l$.}
        \label{fig:van-kampen}
    \end{figure}

    We know that $W'$ is the label of some loop in $\calA_F$ based at $v_0$. We want to show $W$ can also be read off some loop in $\calA_F$ based at $v_0$. We will do this by deleting cells from $D$ in such a way that the cells in the bottom path $\rho$ never get deleted but such that the new top paths still have labels in $L(\calA_F,v_0)$. The number of cells in $D$ can be thought of as a measure of how different $W$ is from $W'$, and the process will terminate in a (degenerate) Van Kampen diagram consisting only of $\rho$. The result will then follow.

    Suppose we have a Van Kampen diagram $D(W',W)$ with $W$ reduced. An \emph{elementary deletion} will be one of the following operations:
    \begin{enumerate}[(a)]
        \item\label{it:remove-spur} \emph{Remove a spur $w$.} If there is a vertex $w$ of degree 1  other than $d_l$ or $d_r$, delete $w$ and the unique edge adjacent to $w$.
        \item\label{it:remove-e} \emph{Remove a free face $e$ with label in $\edges A$.} If there is an edge $e \in \partial D$ that is in the boundary of some 2-cell $C$ with label in $\edges A$, then delete $e$ and the interior of $C$.
        \item\label{it:remove-v} \emph{Remove a free face $e$ with label in $\gen v$ for some $v \in \verts A$.} If there is an edge $e \in \partial D$ that is in the boundary of some 2-cell $C$ coming from a relation of $\bbA_{[v]}$,  then delete $e$ and the interior of $C$.
    \end{enumerate}

    The presence of a spur implies that the word read around $\partial D$ is not freely reduced. Since $W$ is reduced, the spur cannot be covered by the path $\rho$. It follows that if there is a spur other than $d_l$ or $d_r$, $\mu(\rho')$ is not freely reduced. By successively removing spurs and taking new ``top'' boundary paths, we arrive at $\rho''$ such that $\mu(\rho'')$ is the free reduction of $\mu(\rho')$. By Lemma \ref{lem:reduction-sub} there is a loop $\hat \rho''$ in $\calA_F$ based at $v_0$ with label $\mu(\rho'')$.

    If we perform either of the elementary deletions \eqref{it:remove-e} or \eqref{it:remove-v}, we get a new Van Kampen diagram $D'$ which can be seen as a sub-diagram of $D$. Thus, $D'$ also comes with an embedding in the plane and a well-defined boundary word. This is depicted in Figure \ref{fig:deletion}.
    \begin{figure}[htb]
        \centering
        \includegraphics[width=\textwidth]{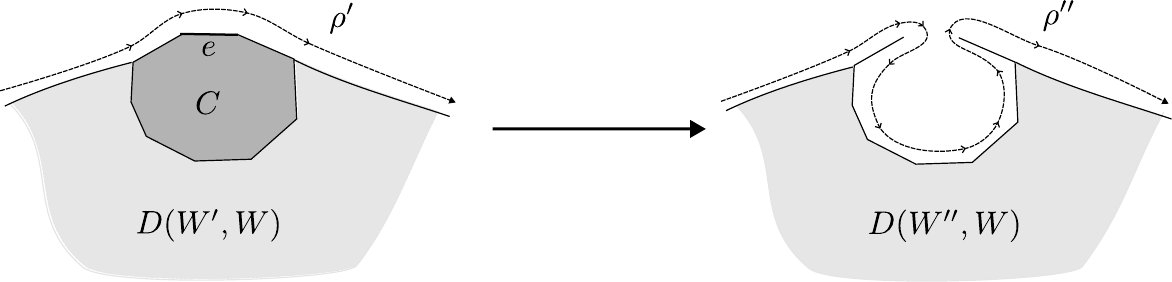}
        \caption{The new top boundary path $\rho''$ after applying elementary deletion \eqref{it:remove-e} or \eqref{it:remove-v}.}
        \label{fig:deletion}
    \end{figure}
    By Lemma \ref{lem:e-substitution} or Lemma \ref{lem:v-substition} (respectively), if the deleted edge $e$ lies in the path $\rho'$, then there is a closed loop in $\calA_F$ based at $v_0$ with label $\mu(\rho'')=W''$. We now proceed as follows:
    \begin{enumerate}
        \item Start with $D=D(W',W)$.
        \item\label{it:step-II} Do the first possible applicable elementary deletion below:
            \begin{enumerate}
                \item If there is a spur besides $d_l$ or $d_r$ remove it. This spur will never lie in the bottom path $\rho$ 
                \item If there is an edge $e$  with label in $\edges A$ that lies in $\rho'$ but not in $\rho$, delete it and the interior of the 2-cell $C$ that contains it using the elementary deletion \eqref{it:remove-e}.
                \item If there is an edge $e$  with label in $\gen v$ for an appropriate $v$ that lies in $\rho'$ but not in $\rho$, and $e$ is contained in a 2-cell $C$ coming from a relation in $\bbA_v$, delete it and the interior of the 2-cell that contains it using the elementary deletion \eqref{it:remove-v}.
            \end{enumerate}
        \item If this gives a new diagram $D(W'',W) \subset D(W',W) \subset \mathbb R^2$ with a new top boundary path $\rho''$, go back to step \ref{it:step-II} and repeat the process. Otherwise, terminate.
    \end{enumerate}

    If this process terminates after $n$ steps with a diagram without any 2-cells, then the diagram must be a tree with spurs $d_l$ and $d_v$, which means that $\rho^{(n+1)}=\rho$, and by repeated application of Lemmas \ref{lem:v-substition}, \ref{lem:e-substitution},  and \ref{lem:reduction-sub} we have that $\mu(\rho^{(n+1)}) = W$ is the label of a closed loop in $\calA_F$ based at $v_0$ as required.

    Suppose now that this process terminates but there are remaining 2-cells. We distinguish two cases.

    \textbf{Case 1:} \textit{One of the remaining 2-cells comes from a relation involving a symbol from $\edges A$.} We will argue that this situation is impossible. Recalling the set of relations $R$ given in \eqref{eqn:BS-eqns}, we see that any such 2-cell $C$ has exactly one pair of edges with label $e\in\edges A$ that we think of as the sides of the 2-cell. The other two paths forming $\partial C$ consist of edges with labels in $\gen{i(e)}$ and $\gen{t(e)}$. Such 2-cells combine to form \emph{$e$-strips} as depicted in Figure \ref{fig:e-strips}.
    \begin{figure}[htb]
        \centering
        \includegraphics[width=\textwidth]{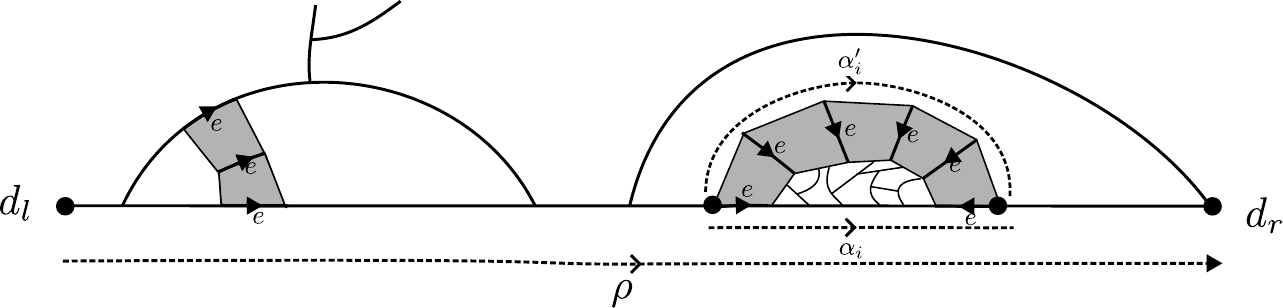}
        \caption{Two $e$-strips and a subpath $\alpha_i$ with a cancellable label.}
        \label{fig:e-strips}
    \end{figure}
    An $e$-strip $\mathcal S \subset D$ is \emph{innermost} if one of the components of $D \setminus \mathcal S$ contains no $e'$-strips for any $e' \in \edges A$.

    \textbf{Claim:} \emph{If $\mu(\rho)$ is reduced in the sense of Definition \ref{defn:reduced}, then there are no $e$-strips whose extremities both lie in $\rho$.} Indeed, if such an $e$-strip existed, then there would be an innermost such strip. Looking at Figure \ref{fig:e-strips}, we see that if this happens, $\rho$ would have a subpath $\alpha_i$ such that there exists another path $\alpha_i'$, whose label has no symbols from $\edges A$, with the same endpoints. In particular the label of $\alpha_i$ is cancellable in the sense of Definition \ref{defn:reduced} contradicting the hypothesis that $W=\mu(\rho)$ is reduced. The claim now follows.

    Now, if the algorithm terminated in Case 1, then this means that every edge in $\rho^{(n+1)}$ is either not contained in a 2-cell, or it is an edge with label in $\gen v$ for some $v \in \verts A$ that is contained in some 2-cell coming from a relation involving a pair of symbols in $e \in \edges A$, i.e. lying in the "side" of an $e$-strip. Such 2-cells assemble into $e$-strips, and in this case, the extremities of all $e$-strips must lie in $\rho$, which contradicts the claim above. It follows that Case 1 cannot occur.

    \textbf{Case 2:} \textit{All remaining 2-cells come from the relations in the $\bbA_v, v\in \verts A$.} In this case, since none of the edges of $\rho^{(n+1)}$ lie in any 2-cells, we must be in a situation as depicted in Figure \ref{fig:baloons}.
    \begin{figure}[htb]
        \centering
        \includegraphics[width=\textwidth]{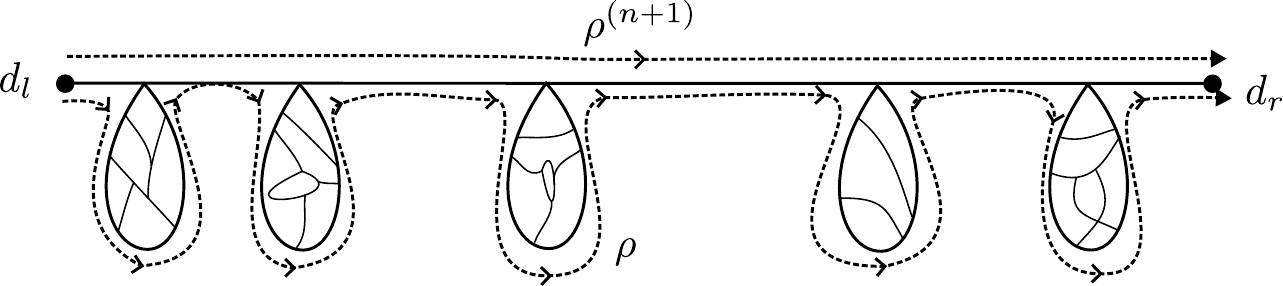}
        \caption{When none of the edges of $\rho^{(n+1)}$ lie in the boundary of a 2-cell.}
        \label{fig:baloons}
    \end{figure}
    A first consequence of Case 2 is that all edges with labels in $\edges A$ must lie in the path $\rho^{(n+1)}$. Furthermore, since both $\mu(\rho)$ and $\mu(\rho^{(n+1)})$ represent elements in $\pi_1(\bbA,v)$, deleting all edges with labels in $\edges A$ separates $D$ in such a way that for every connected component, there is some $v \in \verts A$ such that all edges in that component have a label in $\gen v$. These components also coincide with the $\verts A$-components of $\rho^{(n+1)}$. Let $\alpha_i$ be a $\verts A$-components of $\rho^{(n+1)}$, and consider the subpath $\alpha'_i$ of $\rho$ depicted in Figure \ref{fig:comp-switch}.
    \begin{figure}[htb]
      \centering
      \includegraphics[width=0.5\textwidth]{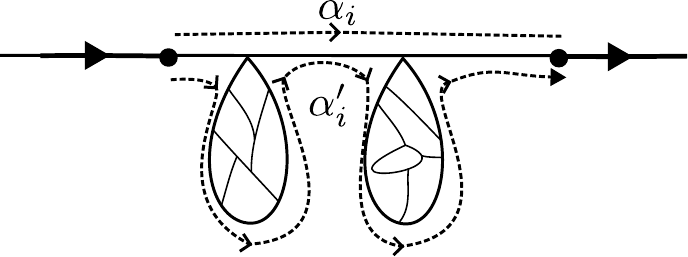}
      \caption{Subpaths of $\rho^{(n+1)}$ and $\rho$ in the same $\verts A$-component with the same evaluations.}
      \label{fig:comp-switch}
    \end{figure}
    
    Seeing as $\alpha_i,\alpha_i'$ bound a van Kampen (sub) diagram, we have $\epsilon(\alpha_i)=\epsilon(\alpha_i')$, and by repeatedly applying Lemma \ref{lem:v-substition}, if $\mu(\rho^{(n+1)})=W^{(n+1)}$ is the label of a based loop in $\pi_1(\calA_F,v_0)$, then so is $W=\mu(\rho)$.

    Since Case 1 and Case 2 exhaust all possibilities, the result follows.    
  \end{proof}

    \begin{proof}[Proof of Theorem \ref{thm:main}]
        The result now follows from Proposition \ref{prop:read-reduced} and Proposition \ref{prop:algo-time}.
    \end{proof}

  \section{Applications}\label{sec:applications}

  \subsection{Reducing words}\label{sec:reduce}
  Given a word $W$ in $(X^{\pm 1})^*$, there is a naive algorithm to pass to a reduced form by reading through $W$ and making the substitution $e^{\pm  1} h e^{\mp 1} \mapsto h'$, where $e \in \edges A$ and where $e^{\pm  1} h e^{\mp 1}= h'$ is a relation of $\Bass(\bbA)$, and then starting over until no more cancellable $\edges A$-letters remain. This algorithm runs in quadratic time in the word length. We propose a faster folding-based algorithm.

  \begin{prop}\label{prop:quick-reduce}
    Suppose that a graph of groups $\bbA$ is explicitly given as in Section \ref{sec:precise} and such that for every vertex group, every element lying in the images of some incident edge group is assigned a symbol as a generator of that vertex group. Then there is an algorithm that takes as input a word $W \in (X^{\pm 1})^*$ representing an element of $\pi_1(\bbA,v)$ and outputs a reduced word $\bar W$ (in the sense of Definition \ref{defn:reduced}) such that $\epsilon(W) = \epsilon(\bar W)$ which runs in time $O(M\log^*(M))$ where $M = |W|$.
  \end{prop}
  
  \begin{proof}
    We note that our requirement on the presentation ensures that all the relations in \eqref{eqn:BS-eqns} have length 4. The algorithm is a slight variation of the folding algorithm. As we present the algorithm, we will provide analysis.

    Start by constructing a based linear directed $X$-labelled graph $(\calW_0,v_0)$ with endpoints $v_0$ and $u_0$ so that the unique simple path $\rho_0$ from $v_0$ to $u_0$ has label $W$. We note that $\pi_1(\calW_0,v_0) = \{1\}$. We now treat $(\calW_0,v_0)$ exactly as the graph constructed in step \ref{it:form-bouquet} of the folding algorithm and proceed to perform the folding algorithm and arrive at a terminal based graph $(\calW_F,v_0)$.

    Following Section \ref{sec:correct}, we see that $\pi_1(\calW_F,v_0) = \pi_1(\calW_0,v_0) = \{1\}$. Recall that we have a continuous map $\calW_0 \to \calW_F$. Following our notation convention, we denote the images $v_0$ and $u_0$ in $\calW_0$ in $\calW_F$ by the same symbols $v_0$ and $u_0$ respectively. Let $\rho$ be a shortest path in $\calW_F$ from $v_0$ to $u_0$. We claim that the label $\mu(\rho)$ is reduced in the sense of Definition \ref{defn:reduced}. Indeed, suppose that was not the case. Then $\rho = \rho_1*\beta*\rho_2$, where $\mu(\beta) = e^{\pm 1}U(\gen{v}^{\pm 1})e^{\mp 1}$ is cancellable. This means that $U(\gen{v}^{\pm 1})$ evaluates to an element in the image of an edge group, so by Lemma \ref{lem:v-substition}, we can replace $\beta$ with $\beta'$ with the same endpoints and where $\mu(\beta')= e^{\pm 1}a_Ue^{\mp 1}$, with $a_U$ being a single letter representing an element of the incident edge group. Because of edge saturation and the argument of Lemma \ref{lem:e-substitution}, we know that we can replace $\beta'$ by $\beta''$, where $\mu(\beta'') = b_U$, which consists of a single symbol, and $e^{\pm 1}a_Ue^{\mp 1}=b_U$ is an identity in $\Bass(\bbA)$. We note that $\rho' = \rho_1*\beta''*\rho_2$ is strictly shorter than $\rho$ but still joins $v_0$ to $u_0$, which is a contradiction. It follows that $\mu(\rho)$ must be reduced.

    Let $\bar\rho_0$ be the image of the path $\rho_0$ in $\calW_0$ in $\calW_F$. Then $\bar\rho_0$ also joins $v_0$ to $u_0$, so $\epsilon(\bar\rho_0*\rho^{-1}) \in \pi_1(\calW_F,v_0)=\{1\}$. Therefore $\epsilon(\rho_0)=\epsilon(\rho)$, and $\mu(\rho)$ is the desired reduced form.

    We finally note that $\rho$ can be found from $\calW_F$ by performing a breadth-first search (BFS) rooted at $v_0$ and directing the tree so that descendants point to predecessors. Since $\calW_F$ is encoded as an adjacency list (each vertex has its list of adjacent edges and each edge has its list of endpoints), the algorithm described in \cite[\S 22.2]{cormen_introduction_2009} and its running time analysis is applicable. Since we also record the edges joining vertices and predecessors in the BFS tree along the way, once $u_0$ is found the path $\rho^{-1}$ can constructed by repeatedly taking predecessors. Thus, the running time analysis in \cite[\S 22.2]{cormen_introduction_2009} tells us that finding $\rho$ can be done in time $O(V+E)$ where $V,E$ represent the number of vertices, edges (respectively) in the graph $\calW_F$.\footnote{Note that here, since we de not need to calculate distance, we do not need to do any arithmetical operations. Thus the algorithm truly only uses $O(V+E)$ pointer operations.} For a fixed $\bbA$, $V+E$ is bounded above by a constant multiple of $M$ and the result follows.

  \end{proof}

\begin{cor}\label{cor:vf_fast}
    Let $G=\langle Y |S \rangle$ be fixed presentation of a virtually free group. Then there exists an algorithm that takes as input a word $U \in (Y^\pm 1)^*$ and a tuple of words $(W_1,\ldots,W_n)$ in $(Y^{\pm 1})^*$ and decides whether $\epsilon(U) \in \langle \epsilon(W_n),\ldots,\epsilon(W_n)\rangle$ in time $O(N \log^*(N))$ where $N = |U|+|W_1|+\cdots+|W_n|$.
  \end{cor}

\begin{proof}
    By hypothesis, $G$ is isomorphic to the fundamental group $\pi_1(\bbA,v_0)$ of a graph of groups with finite vertex groups. Let $\bbA$ be explicitly given as in Section \ref{sec:precise}. We can also require it to fulfill the hypotheses of the statement of Proposition \ref{prop:quick-reduce} as this can be achieved by obvious Tietze transformations. Such an isomorphism $\phi:G \to \pi_1(\bbA,v_0)$ is induced by a mapping $\phi|_Y:Y \to \pi_1(\bbA,v_0)\leq \Bass(\bbA)$. It follows that if such a $\phi|_Y$ is given, then setting $K = \max\left\{\left|\phi|_X(x)\right|_{Z} \mid x \in X\right\}$, where $X$ a generating set for $\Bass(\bbA)$ as in the statement of Theorem \ref{thm:main}, by substituting the symbols in $W \in (Y^{\pm 1})^*$ by the appropriate words in $(X^{\pm 1})^*$ we find that the image $\phi(\epsilon(W))$ can be effectively be represented as a word in $(X^{\pm 1})^*$ of length at most $K|W|$.
    
    The remaining issue is that this rewriting may not be freely reduced and, even after free reductions, may not be reduced in the sense of Definition \ref{defn:reduced}. By Proposition \ref{prop:quick-reduce} there is an algorithm that will produce a reduced form for the image of $W$ in time $O(K|W|\log^*(K|W|))$. The result now follows immediately from Theorem \ref{thm:main} and the fact that $\log^*(K|W|) < \log^*(K)+\log^*(|W|)$
\end{proof}

  \subsection{Detecting freeness and equality of subgroups}\label{sec:detect}

This proof of the next result uses actions on Bass-Serre trees.    
    
    \begin{prop}\label{prop:detect-free}
      $H \leq \pi_1(\bbA,v_0)$ is a free group if and only if every $\verts A$-component $W$ is isomorphic to a copy of the Cayley graph $\Gamma_{[w]}$, where $w \in \verts W$. Furthermore, if $H$ is given by a tuple $(W_1,\ldots,W_s)$ of words in the generating set of $\Bass(\bbA)$, then  freeness of $H$ can be determined by modifying the folding algorithm in a way that only incurs an $O(N)$ overhead cost where $N$ is the sum of the lengths of the input words.  
    \end{prop}

    \begin{proof}
      An element $g\in \pi_1(\bbA,v)$ fixes a point in the dual Bass-Serre tree if and only if it is conjugate to one of the vertex groups $\bbA_v$ for some $v \in \verts A$. Thus, $H$ contains such an element if and only if there is a loop in $(\calA_H,v_0)$ with reduced label $W f W^{-1}$, where $f \in \bbA_v\setminus \{1\}$. It follows that if $\rho$ is a path in $(\bbA,v)$ starting at $v$ with label $W$, then there must be a loop at the endpoint of $\rho$ with label $f$. Thus  $W f W^{-1}$ is non-trivial if and only if $f \neq 1$, and such an element exists if and only if some $\verts A$-component of $\calA_H$ is not a Cayley graph of a vertex group, but rather a proper quotient of such a Cayley graph so that it contains a loop with a label that doesn't evaluate to the identity.

      Since $\bbA$ is a graph of finite groups, if some $g \in H$ fixes a vertex in the Bass-Serre tree, then that element has finite order and $H$ is not free. If there are no such elements, then $H$ acts freely on $T$ and therefore is free. We now prove the second statement.

      Whether or not $H$ is free can be verified by modifying the folding algorithm as follows: create a new list called \texttt{VERIFICATION} and append to it every vertex of $\calA_0$, we can modify data structures so that the same objects lie in multiple lists. Once the folded graph $\calA_F$ is produced, we go through the list \texttt{VERIFICATION}. For each vertex in the $v$ list, find its image $\bar v$ in $\calA_F$ (this involves a \emph{find} operation, see Section \ref{sec:orig-algo-description}), attach a copy of $\Gamma_{[v]}$ to $\bar v$ and fold (this takes $O(1)$ operations). Again we can modify the data structure so that if two vertices of vertices of the copy of $\Gamma_{[v]}$ we just added get identified we break the loop and declare $H$ to be non-free, the extra cost per vertex in the list is again $O(1)$ disjoint set and doubly linked list operations. Since this is repeated at most $N$ times (the initial length of the list \texttt{VERIFICATION})  determining whether $H$ is free involves an $O(N)$ cost overhead.
    \end{proof}

    An immediate application of the Theorem \ref{thm:main} and Proposition \ref{prop:quick-reduce} is:
    
    \begin{prop}\label{prop:detect-equal}
      Let $(K_1,\ldots, K_r)$ and $(W_1,\ldots,W_s)$ be two tuples of words in $X$. Then there is an algorithm that runs in time  $O(N \log^*(N))$, where $N$ is the total length of all the words in the input and determines whether the subgroups $\bk{K_1,\ldots, K_r}$ and $\bk{W_1,\ldots,W_s}$ are equal.
    \end{prop}
    \begin{proof}
      Using Proposition \ref{prop:quick-reduce} in time $O(N \log^*(N))$, we can replace all the $K_i,W_j$ by reduced forms. We now perform the folding algorithm, again in time $O(N \log^*(N)$. Once a folded graph is constructed, determining whether some reduced word of length $n$ represents an element of the subgroup takes time $O(n)$. Since the sum of the lengths of the $K_i,W_j$ add up to $N$, verifying whether $K_i \in \bk{W_1,\ldots,W_s}$ and $W_j \in \bk{K_1,\ldots, K_r}$ for all $K_i,W_j$, takes an additional $O(N)$. This finally establishes whether or not both inclusions \[\bk{K_1,\ldots, K_r} \leq \bk{W_1,\ldots,W_s} \textrm{~and~}  \bk{W_1,\ldots,W_s} \leq \bk{K_1,\ldots, K_r}\] hold.
    \end{proof}
    
    \subsection{The proof of Theorem \ref{thm:intro}}\label{sec:gl2z}

We have that $\GL{2,\mathbb Z} \approx D_6*_{D_2}*D_4$, where $D_n$ denotes the dihedral group with $n$ elements, which we can express as a graph of groups where the underlying graph $A$ has vertices $u,v$ and a single edge $e$, with $i(e)=u$ and $t(e)=v$. We take $\bbA_u \approx D_6$ and $\bbA_v \approx D_4$. Now, following \cite[$\S$I.5]{dicks_groups_1989}, we  find 
$A=\begin{pmatrix}1&-1\\0&-1\end{pmatrix} \in D_6, B=
\begin{pmatrix}
  1&0\\0&-1
\end{pmatrix} \in D_4$, and $C=
\begin{pmatrix}
  0&1\\1&0
\end{pmatrix} \in D_6\cap D_4 = D_2$. Expressing these as elements of $\pi_1(\bbA,v)$, we write $A=a\in \bbA_u$, $B=ebe^{-1}, b \in \bbA_v$ and $C=c = ec'e^{-1}$ where $c\in \bbA_u$ and $c'\in \bbA_v$. Just as in Section \ref{sec:sl2z-eg}, we can find a presentation $\Bass(\bbA) = \langle X | R \rangle$ satisfying the requirements of Section \ref{sec:precise} where $a \in \gen u \subset X \supset \gen v \ni b$.

\begin{lem}\label{lem:matrix-length}
    Let \[M = \begin{pmatrix}
        a&b\\c&d
    \end{pmatrix} \in \GL{2,\bb Z}\] and let $\GL{2,\mathbb Z} = \pi_1(\bbA,u)$ where $\bbA$ is as given above and $\Bass(\bbA)=\langle X | R \rangle$ conforms to the requirements set out in Section \ref{sec:precise}. Then $M$ can be represented as a word of length $O(|a|+|b|+|c|+|d|)$ in $(X^{\pm 1})^*$
\end{lem}

\begin{proof}
We denote $BA = E =
\begin{pmatrix}
  1&-1\\
  0&1
\end{pmatrix}$, and take the inverse $E^{-1}=
\begin{pmatrix}
  1&1\\0&1
\end{pmatrix}$. As elements of $\pi_1(\bbA,u)$, we write \[E=ebe^{-1}a, E^{-1}=a^{-1}eb^{-1}e^{-1},\]
where $a, b$ as defined above.

Multiplying $M$ by $E^{\pm 1}$ on the left will correspond to adding or subtracting the second row from the first, and multiplying $M$ by $C$ on the left will correspond to swapping the rows of $M$. Thus these are sufficient to implement integer row reduction.  We row-reduce $M$ to the identity matrix while keeping track of the number steps required.

The first objective is to use the Euclidean algorithm to find some matrix $N_1$ such that:\[
  N_1
  \begin{pmatrix}
    a&b\\c&d
  \end{pmatrix}
  =
  \begin{pmatrix}
    \gcd(a,c) & n \\ 0 &u
  \end{pmatrix},
\] for some $n,u$. If $a$ or $c$ are zero then $N_1 = I_2$ or $C$ and we can skip ahead to \eqref{eqn:upper_triangular} below. Otherwise, up to multiplying on the left by some product length at most 3 of $B,C ,-I_2$, we can assume that $0<a<c$. By the division algorithm, there is $0\leq r < a$ such that $c = aq+r$.

Take $F = CEC =
\begin{pmatrix}
 1&0\\-1&1 
\end{pmatrix}.$ We have\[
  F^qM=
  \begin{pmatrix}
    a & b\\
    c-qa &d-qb
  \end{pmatrix} =
  \begin{pmatrix}
    a&b\\r&d'
  \end{pmatrix}.
\] We now consider the  absolute values of the entries of the resulting matrix. The entries $a$ and $b$ will be unchanged in $F^qM$, and by hypothesis, $0 \leq c-qa < a < c$. Furthermore, we have that $c-r =qa \geq q$. Thus the sum of the absolute values of the entries in the first column decreased by at least $q$ after multiplication by $F^q$.

We now look at the second column. It suffices to compare $|d|$ to $|d'|$. On the one hand, we have\[
  ad-bc = \pm 1 \textrm{~and~} ad'-br=\pm 1,
\] since $\det(F^qM)= \pm 1$. This gives  $|d'| \leq 1+ \left|\frac{br}{a}\right|$ and $|\frac{bc}{a}|-1 \leq |d|$, and since $|a|>0$ and $|c|\geq |r|$, we have that $|d'| \leq |d|+2$. So the sum of the absolute values of the entries in the second column increased by at most 2.

To continue the Euclidean Algorithm, we would multiply by $C$ on the left to interchange the rows and repeat the process until we get the desired matrix 
\[N_1 = F^{q_s}C\cdots C F^{q_1}.\] Repeating the argument given above we see that in passing from $M$ to $N_1M$ the sum of the absolute values of the entries in the first column decreased by at least $q_1+\cdots+q_r$ thus we have\[
s \leq \sum_{i=1}^s q_i \leq |a|+|c|
\] and it follows that $N_1$ can be represented by a word in $(X^{\pm 1})^*$ of length $O(|a|+|c|)$. We now have \begin{equation}\label{eqn:upper_triangular}N_1M=\begin{pmatrix}
    \gcd(a,c) & n \\ 0 &u
  \end{pmatrix} \in \GL{2,\mathbb Z},
\end{equation}
thus $\gcd(a,c),u \in \{-1,1\}$. As for the top right entry $n$, note that by repeating the argument above, we see that the absolute values of the entries in the second column increased by at most $2s \leq 2|a|+2|c|$ which easily gives \[n \leq \max\{|b|,|d|\}+2|a|+2|c| \leq 2(|a|+|b|+|c|+|d|).\] So, after perhaps multiplying by $B$ to make the bottom right entry positive, we have\begin{equation}\label{eqn:tight}
    E^{n}
    \begin{pmatrix}
      \pm 1 & n\\ 0 &1
    \end{pmatrix}=
    \begin{pmatrix}
      \pm 1 & 0 \\ 0& 1
    \end{pmatrix} \in D_4.
  \end{equation}

  It follows that a matrix $M=
  \begin{pmatrix}
    a&b\\c&d
  \end{pmatrix} \in\GL{2,\mathbb Z}$ can be expressed as a word $W \in (X^{\pm 1})^*$  representing an element of $\pi_1(\bbA,u)$ with $|W| = O(|a|+|b|+|c|+|d|)$ as required.
\end{proof}

We note that \eqref{eqn:tight} tells us that this linear upper bound is also a lower bound so Lemma \ref{lem:matrix-length} is sharp. Finally, we can prove the first stated theorem.

\begin{proof}[Proof of Theorem \ref{thm:intro}]
    By Lemma \ref{lem:matrix-length} we can represent the matrices in $\GL{2,\bb Z}$ as $\bbA$-loops based at $u$ whose length is at most a multiple of the sum of the absolute values of the coefficients. Furthermore, the number of arithmetic operations needed is also a multiple of the sum of the absolute values of the coefficients. The result now follows from Proposition \ref{prop:quick-reduce}, and Theorem \ref{thm:main}.
\end{proof}

\bibliography{SFPVF}
\bibliographystyle{alpha}

\end{document}